\documentclass[12pt]{amsart}
\usepackage{amsmath,graphics,amsfonts,amsthm}
\usepackage {amssymb}
\usepackage{mathrsfs}
\usepackage{inputenc}
\usepackage{tikz-cd} 
\usepackage{xypic}
\usepackage{mathtools}
\usepackage{soul}
\usepackage{extarrows}
\usepackage[left=1.5in,right=1.5in,top=1.5in,bottom=1.5in]{geometry}
\usepackage{hyperref}
\usepackage[pagewise]{lineno}

\theoremstyle{plain}
\newtheorem{theorem}{Theorem}[section]
\newtheorem{corollary}[theorem]{Corollary}
\newtheorem{lemma}[theorem]{Lemma}

\newtheorem{proposition}[theorem]{Proposition}
\newtheorem{definition-lemma}[theorem]{DefiniCorollarytion-Lemma}
\theoremstyle{remark}
\newtheorem{remark}[theorem]{Remark}
\newtheorem{claim}[theorem]{Claim}

\theoremstyle{definition}

\theoremstyle{definition}
\newtheorem{definition}[theorem]{Definition}
\theoremstyle{definition}

\numberwithin{equation}{section}

\def\Sing{\operatorname{Sing}}
\def\Stab{\operatorname{Stab}}
\def\Orb{\operatorname{Orb}}
\def\Bim{\operatorname{Bim}}
\def\Aut{\operatorname{Aut}}
\def\Supp{\operatorname{Supp}}
\def\Proj{\operatorname{Proj}}

\def\Diff{\operatorname{Diff}}
\def\rddelta{\lfloor \Delta\rfloor}
\def\rdgamma{\lfloor\Gamma\rfloor}
\def\Ex{\operatorname{Ex}}
\def\lct{\operatorname{lct}}
\def\PA{\operatorname{PA}}
\def\A{\operatorname{A}}
\def\MA{\operatorname{MA}}

\def\MPA{\operatorname{MPA}}
\def\Spec{\operatorname{Spec}}

\begin{document}
\title[ abundance for K\"ahler threefolds]{Abundance for semi log canonical  compact K\"ahler threefolds}
\author{Swapnajit Das}
\address{School of Mathematics, Tata Institute of Fundamental Research, Homi
Bhabha Road, Navy Nagar, Colaba, Mumbai 400005, India}
\email{das@math.tifr.res.in}

\begin{abstract}
    In this article we show that if $(X,\Delta)$ is a compact K\"ahler slc pair of dimension $3$ such that $K_X+\Delta$ is nef, then $K_X+\Delta$ is semiample. Moreover, we extend our result to a special subclass of Fujiki class $\mathcal{C}$ varieties, and show that in this subclass, if $(X, \Delta)$ is a slc pair and $K_X+\Delta$ is nef but not numerically trivial, then it is semiample. Moreover, we also construct a slc pair $(X, \Delta)$ in Fujiki class $\mathcal{C}$ such that $K_X+\Delta \equiv 0$ but $K_X+\Delta$ is not semiample. In particular, the slc abundance fails for Fujiki class $\mathcal C$ varieties.

\end{abstract}

\maketitle

\tableofcontents

\section{Introduction}\label{sec:int}

The Minimal Model Program (MMP) plays a central role in the birational classification of projective varieties. In dimension three, the MMP was completed in the 1990s. The abundance theorem, which states that a nef canonical divisor is semiample, was proved for projective threefolds by  Miyaoka, Kawamata, Keel, Matsuki and McKernan \cite{Miy87,Miy88a,Miy88b}, \cite{Kaw92}, \cite{KMM94, KMM04}.

The MMP was later extended to compact K\"ahler threefolds (see \cite{HP15,HP16,DH25}), providing an analogous bimeromorphic classification in the analytic setting. The log abundance theorem for compact K\"ahler threefolds has recently been established by Campana-H\"oring-Peternell \cite{CHP16}, and  O. Das and W. Ou (see \cite{DO24,DO25}).\\
\begin{theorem}\cite{CHP16,DO24,DO25}\label{abun}
    Let $(X,\Delta)$ be a lc pair such that $X$ is a compact K\"{a}hler variety of dimension $3$. If $K_X +\Delta$ is nef, then it is semiample.
    \end{theorem}
 In \cite{Fu00}, Fujino extended the abundance theorem for semi log canonical projective threefolds. 

\begin{theorem}\cite[Theorem 0.1]{Fu00}
    Let $(X, \Delta)$ be a projective slc pair of dimension $3$. If $K_X+\Delta$ is nef, then it is semiample. 
\end{theorem}

 In this article, we prove the following:

 \begin{theorem}\label{thm:mainimp}
    Let $(X,\Delta)$ be a compact K\"ahler slc pair of dimension $3$ such that $K_X+\Delta$ is nef. Then it is semiample.\\
    \end{theorem}
   This result is proved in Page No. \pageref{proof:mainthm}. 
Moreover, as an application, we prove the following.
\begin{corollary}\label{cor:mainimp}
     Let $(X,\Delta)$ be a connected, compact analytic Normalized-K\"ahler slc pair (see Definition \ref{def:NK}) of dimension $3$  such that $K_X+\Delta$ is nef. Then either $K_X+\Delta$ is semiample or it is numerically trivial, i.e. $K_X+\Delta \equiv 0$.
\end{corollary}

\noindent
\textbf{Fujino's strategy and our Idea:} One of the main ingredients of Fujino's proof in \cite{Fu00} is the finiteness of pluricanonical representation of log canonical pairs in dimension 2 (see \cite[Theorem 3.5]{Fu00}). This depends on a classical finiteness theorem on Moishezon manifolds due to Ueno (see \cite[Theorem 14.10, Page No. 182]{KU75}). \\
 Let $(X,\Delta)$ be a projective slc threefold and  $(Z,\Gamma)\to (X, \Delta)$ its dlt modification. Then using the finiteness of $B$-pluricanonical representation in dimension $2$, Fujino first proves that the set of all \emph{admissible} sections (see Definition \ref{def:pre&ad}) generate the line bundle $\mathcal O_{\rdgamma}(m(K_Z+\Gamma)|_{\rdgamma})$ (see \cite[Lemma 4.7]{Fu00}). Then using \cite[Proposition 4.5]{Fu00}, he shows that these admissible sections lift to the \emph{preadmissible} sections on $Z$ and generate the line bundle $\mathcal{O}_Z(m(K_Z+\Gamma))$. Then according to \cite[Lemma 4.2]{Fu00}, all preadmissible sections of $\mathcal{O}_Z(m(K_Z+\Gamma))$ descend to $X$, which then gives the global generation of $\mathcal O_X(m(K_X+\Delta))$. 
 
In the K\"ahler category, the finiteness of the pluricanonical representation fails for certain non-algebraic K3 surfaces (see Remark \ref{K3} for details). On the other hand, it turns out that if a compact K\"ahler surface has positive (log) Kodaira dimension, then the pluricanonical representation is finite (see Theorem \ref{plu}). However, this restricted finiteness result is not sufficient to implement Fujino’s general strategy outlined above.

To overcome this limitation, we introduce two new notions $-$ \emph{minimally admissible} and \emph{minimally preadmissible} sections (see Definition \ref{def:ma-mpa}), which generalize Fujino’s notions of admissible and preadmissible sections (see Definition \ref{def:pre&ad}). By working with these refined classes of sections, we are able to adapt Fujino’s strategy. One of the main difficulties in the proof of Theorem \ref{thm:mainimp} is lifting admissible sections for the components of the normalization of $X$ which has log-Kodaira dimension $0$. This poses a significant technical challenge. \\ 
We prove our main theorem \ref{thm:mainimp} in several steps. First we show if every irreducible component of normalization has positive Kodaira dimension, then abundance holds (See Corollary \ref{main}). Then we show if at least one component has positive Kodaira dimension, then also abundance holds (See Theorem \ref{thm:abun-non-numerical}). In the last step, we show abundance holds as well in numerically trivial case (See Theorem \ref{mainthm:num-trivial}). In this particular step, we explicitly use the existence of a K\"ahler class on $X$. Except the last step, the arguments of the previous two steps also work as well in case of Normalized-K\"ahler pairs (see Definition \ref{def:NK}) which helps us to prove  Corollary \ref{cor:mainimp}. Failure of finiteness of pluricanonical representation and non existence of K\"ahler class led us to the following counterexample. In fact the idea of the proof in numerically trivial case for slc K\"ahler threefold is motivated from this counterexample.
\begin{theorem}\label{counterexample}
 There exists an irreducible compact analytic variety  $X$ of dimension $3$ in Fujiki class $\mathcal C$ with slc singularities  such that $K_X$ is nef and $\nu(\tilde X, K_{\tilde X}+\tilde D)=\kappa(\tilde X, K_{\tilde X}+\tilde D)=0$, where $\mu:(\tilde X, \tilde D)\to X$ is the normalization morphism, but $K_X$ is not semiample. More specifically, we have $H^0(X, mK_X)=0$ for all sufficiently large and divisible positive integer $m>0$.\\
 
\end{theorem}

This result is proved in Section \ref{sec:slc-abundance}, Page No. \pageref{proof:counterexample}.\\

\textbf{Development of the project:}
In our earlier version of this paper, we claimed the following three main results:
\begin{enumerate}
    \item Abundance holds for compact K\"ahler sdlt pairs of dimension $3$.
    
    \item Abundance does not hold in general for compact K\"ahler slc
    threefolds. More precisely, we claimed the existence of an
    irreducible compact K\"ahler slc threefold $X$ such that $K_X\equiv 0$ but not semiample.
    
    \item Let $(X,\Delta)$ be a compact K\"ahler slc pair, and let
    $
        \mu\colon
        (X',\Delta'+D')
        =
        \bigsqcup_{i=1}^{k}
        (X'_i,\Delta'_i+D'_i)
        \longrightarrow
        (X,\Delta)
    $
    be its normalization. If
    $
        \kappa\bigl(X'_i,
        K_{X'_i}+\Delta'_i+D'_i\bigr)\geq 1
    $
    for every $i\in\{1,\ldots,k\}$, then abundance holds for
    $(X,\Delta)$.
\end{enumerate}

Soon after the first version of this paper appeared on arXiv, Prof. J\'anos
Koll\'ar pointed out that the variety $X$ constructed in (2) admits an \'etale double cover
$
    f\colon Y\longrightarrow X
$
such that $Y$ has sdlt singularities. At that stage, $X$ was assumed to
be K\"ahler. Under this assumption, $Y$ would also be K\"ahler, and the
failure of abundance for $X$ would imply the failure of abundance for
$Y$. Thus, Prof. Koll\'ar's observation showed that our claim (1) was
incompatible with the proposed counterexample in  (2).

The positive result in claim (3), however, was unaffected by this
observation. Prof. Koll\'ar then asked whether claim (3) could be refined: namely, whether abundance holds for a compact connected
K\"ahler slc threefold $(X,\Delta)$ whenever $K_X+\Delta$ is nef and
not numerically trivial. In other words, he asked whether the claim $(3)$ can be extended to the case when $\kappa(X'_i, K_{X'_i}+D'_i+\Delta'_i)\geq 1$ for at least one $i\in\{1,2,...,k\}$ where $\mu: X'\to X$ is the normalization. We started working on this more
general situation and proved that abundance indeed holds in the
non-numerically trivial case (see
Theorem \ref{thm:abun-non-numerical}).

Approximatey two weeks after we proved Theorem \ref{thm:abun-non-numerical},  Zhiyuan Jiang pointed out that the variety $X$ occurring
in our proposed counterexample in (2) is in fact not K\"ahler, although its
normalization is K\"ahler, and hence $X$ lies in the Fujiki class $\mathcal C$ varieties (see Remark \ref{rmk:jiang}). We had a discussion and observed that if $X$ admitted a K\"ahler class, it would force a K\"ahler form to be invariant under the gluing automorphism, contradicting the fact that this automorphism has infinite order.
This observation changed the situation in the numerically trivial case
completely. The example from the earlier version no longer provided a
counterexample in the K\"ahler category. We were therefore led to
reconsider the numerically trivial case from a positive point of view.
 We realized that invariance of K\"ahler class under $\sigma$ is not essential where $\sigma$ is the gluing automorphism of the $K3$ surface $S$. Rather, it is enough to satisfy that the difference $\sigma^*\omega-\omega\in NS_{\mathbb{R}}$. More precisely, we establish Lemma \ref{lem:preservation-kahler}. Using this lemma, we prove
Theorem \ref{mainthm:num-trivial} and establish abundance also when
$K_X+\Delta$ is numerically trivial.

Combining the numerically trivial case with
Theorem \ref{thm:abun-non-numerical}, we obtain abundance for compact
K\"ahler slc pairs of dimension $3$ in full generality. Thus, the
development of the paper led from the proposed counterexample in the
earlier version to a substantially stronger positive result.

Finally, although the example in (2) constructed in the earlier version is
not K\"ahler and hence does not give a counterexample to slc abundance in
the K\"ahler category, it remains a valid counterexample to slc abundance in
the Fujiki class $\mathcal C$. On the positive side, we also establish abundance for
a special subclass of varieties in the Fujiki class $\mathcal C$ (see
Corollary \ref{cor:mainimp}).\\

The article is organized in the following manner: In Section \ref{sec:premin} we prove some preliminary results which will be used throughout the article. In Section \ref{sec:finiteness} we prove our special finiteness result for compact K\"ahler lc pair of dimension $2$. In Section \ref{sec:ad-and-pread} we prove some properties of admissible and preadmissible sections similar to those in \cite[Section 4]{Fu00} in the K\"ahler settings. Section \ref{sec:sdlt-abundance} is one of our main sections of this article. Here we define minimally admissible and preadmissible sections and some results related to these sections. We prove our main Theorem in non numerically trivial case  in Section \ref{sec:mainthm}. We prove our main theorem in numerically trivial case in Section \ref{sec:main result}. At last, we prove Theorem \ref{counterexample} in Section \ref{sec:slc-abundance}.\\

\textbf{Acknowledgments}: The author would like to express his sincere gratitude to Professor Omprokash Das for suggesting this problem and for his constant guidance, encouragement, and valuable discussions throughout the development of this work. The author is also grateful to Prof. Christopher Hacon, Prof. Osamu Fujino, Prof. Paolo Cascini, and Prof. James McKernan for their careful reading of an earlier draft of the manuscript and for their helpful comments and suggestions. The author is particularly indebted to Prof. J\'anos Koll\'ar for identifying an error in a previous arXiv version of this article and for the many insightful email exchanges that followed. The author is grateful to  Zhiyuan Jiang for pointing out that the counterexample in his earlier version of this article does not lie in the K\"ahler category and for fruitful discussions over Zoom. Finally, the author wishes to thank Mr. Arnab Roy for several helpful discussions.

\section{Preliminaries}\label{sec:premin}
In this section we collect some definitions and results which will be used throughout the article. 
\begin{definition}
    Let $X$ be a reduced complex space of pure dimension satisfying the following two properties :
    \begin{enumerate}
        \item $X$ is $S_2$.
        \item $X$ is normal crossing in codimension $1$.
    \end{enumerate}
    Let $\Delta \geq 0$ be an effective $\mathbb{Q}$-divisor on $X$ such that the support $\Delta$ does not contain any component of the conductor divisor of $X$, and $K_X+\Delta$ is $\mathbb{Q}$-Cartier.
    Let $X=\cup X_i$ be the unique decomposition into irreducible components and $\mu : X':= \sqcup X'_i \to X$ the normalization, where $\sqcup$ denotes the disjoint union. Let $\Theta$ and $\Theta_i$ be  $\mathbb{Q}$-divisors on $X'$ and $X'_i$, respectively, such that $K_{X'}+\Theta\sim_{\mathbb{Q}}\mu^*(K_X+\Delta)$ and $K_{X'_i}+\Theta_i \sim_{\mathbb{Q}}(K_{X'}+\Theta)|_{X'_i}$ for all $i$.\\
    We say that $(X,\Delta)$ is \emph{semi log canonical} or \emph{slc},  if $(X',\Theta)$ is log canonical.\\
    We say that $(X,\Delta)$ is a \emph{semi divisorial log terminal} or \emph{sdlt}, if $X_i'$ is isomorphic to $X_i$ for all $i$, and $(X',\Theta)$ is divisorial log terminal.\\
    \begin{definition}\label{def:NK}
        Let $(X,\Delta)$ be a compact analytic pair. $X$ is called Normalized-K\"ahler if its normalization $X'$ is K\"ahler. Note that this forms a subclass of the Fujiki class. We call $(X,\Delta)$, a Normalized-K\"ahler slc pair if $X$ is Normalized-K\"ahler and $(X,\Delta)$ has slc singularities.
    \end{definition}
    
    \begin{remark}
    Let $\mu:X' \to X$ be the normalization, and $D \subset X$ and $D' \subset X'$ the conductor divisors on $X$ and $X'$ respectively. Both $D$ and $D'$ are reduced divisors by \cite[Definition-Lemma 5.10]{Kol13}. Then $K_{X'}+D'+\Delta' \sim_{\mathbb{Q}} \mu^*(K_X+\Delta)$ where $\Delta'$ is the divisorial part of $\mu^{-1}(\Delta)$. In particular, $D'+\Delta'=\Theta$ as defined above.
    \end{remark}
    \begin{remark}\label{rmk:sdlt}
Let $(X,\Delta)$ be a sdlt pair. Since $X_i'$ and $X_i$ are isomorphic (by definition above), each $X_i$ is normal. The Non-normal locus of $X$ is exactly the union of intersections of irreducible components. In particular, the conductor divisor $D$ on $X$ is given by $\sum_{i\neq j}(X_i \cap X_j)$ and the conductor divisor $D'=\mu^{-1}{(D)}$ on $X'$.
    \end{remark}
\end{definition}
The following lemma tells us how to run relative MMP for a morphism of K\"ahler varieties when the base is projective.
\begin{lemma}\label{MMP}
   Let $f:X \to Y$ be a  surjective morphism with connected fibers. Let $\Delta \geq 0$ be an effective divisor such that $(X,\Delta)$ is a compact K\"{a}hler dlt pair of  dimension at most  $3$ and $Y$ is a projective variety. Then we can run  a terminating $K_X+\Delta$-MMP over $Y$.
   \end{lemma}
\begin{proof}
Let $(X,\Delta)$ be a  compact K\"{a}hler dlt pair of dimension at most $3$. Let $R$ be a $K_X+\Delta$-negative extremal ray over $Y$. By cone theorem \cite[Theorem 5.2]{DHP25}, we can choose a curve $C$ such that  $0<-(K_X+\Delta)\cdot C\leq 6$ where $R=\mathbb{R}_{\geq0}[C]$. We know that MMP can be run for compact K\"ahler dlt pair of dimension at most 3 by \cite{DH25}.\\

Since $Y$ is projective, we can choose a very ample Cartier divisor $H$ on $Y$ such that $H\cdot \Gamma\geq 7$ for every curve $\Gamma\subset Y$. As $H$ is very ample, the linear system $|f^*H|$ is base point free.  By Bertini's theorem, we can choose a divisor $D\in |f^*H|$ such that $(X,\Delta+D)$ is still dlt. For every vertical curve $C$, $(K_X+\Delta+D)\cdot C=(K_X+\Delta)\cdot C$ and hence $(K_X+\Delta+D)\cdot C<0$. Let $g_{R}:X \to Z$ be the contraction of $(K_X+\Delta+D)$-negative extremal ray $R=\mathbb{R}_{\geq0}[C]$.\\
Let $C_1$ be any curve contracted by $g_R$. Then $C_1$ is also contracted by $f$. Indeed if this is not the case, let $C_2=f(C_1)$. 
Then we have
\[ (K_X+\Delta +D)\cdot C_1=(K_X+\Delta)\cdot C_1 + H\cdot C_2 \geq -6+7 \]
which gives a contradiction.\\

Since every vertical curve for $g_{R}$ is also a vertical curve for $f$ and both $f$ and $g_{R}$ have connected fibers, it follows from Lemma \ref{rig}, that there exists a morphism $f_1:Z \to Y$ such that $f_1\circ g_{R}=f$. Finally since a $K_X+\Delta+D$-MMP is also a $K_X+\Delta$-MMP, we obtain either a minimal model or a Mori fiber space over $Y$ by \cite[Theorem 1.1 and 1.2]{DH25}.
\end{proof}
\begin{lemma}\label{lem:Top}
   Let $X$ and $Y$ be two topological spaces and let $f:X\to Y$ be a closed surjective map. Assume that  $Y$ is connected and that for every $y\in Y$, $f^{-1}(y)$ is connected . Then $X$ is connected. In particular, number of connected components of $X$ and $Y$ coincide.
   \end{lemma}
\begin{proof}
Suppose, for contradiction, that $X$ is not connected. Assume that $X$ has two connected components, say $X_1$ and $X_2$. Since $f$ is surjective, we have $Y=f(X_1)\cup f(X_2)$. Since $f$ is closed, both $f(X_1)$ and $f(X_2)$ are closed. Since $Y$ is connected, $f(X_1) \cap f(X_2)$ must be non empty. Choose $y\in f(X_1) \cap f(X_2)$. Then $f^{-1}(y)$ intersects  both components $X_1$ and $X_2$. Since $f^{-1}(y)$ is connected, it must be contained in one connected component. This is a contradiction and hence $X$ is connected.
\end{proof}
\begin{lemma}\label{N}
    Let $(X,\Delta)$ and $(Z,\Delta')$ be two compact normal analytic varieties with  $K_X+\Delta$ and $K_{Z}+\Delta'$ $\mathbb{Q}$-Cartier. Let $f:X \to Y$ and $g:Z \to Y $ be surjective morphisms between normal analytic varieties and $p:X \dashrightarrow Z$  a bimeromorphic map such that $f=g\circ p$. Assume that 
    \begin{enumerate}
        \item $p^{-1}$ has no exceptional divisor, 
        \item $\Delta'=p_*\Delta$, and
        \item $K_X+\Delta\equiv_f 0, K_Z+\Delta' \equiv_g 0 $.
    \end{enumerate}
    Then $p$ is a $B$-bimeromorphic map (see Definition \ref{def:b-bimeromorphic}).
    \begin{proof}
        Let $W$ be a common resolution that resolves $p$ and satisfies the following commutative diagram:

\[\begin{tikzcd}
	& W \\
	X && Z \\
	& Y
	\arrow["\alpha"', from=1-2, to=2-1]
	\arrow["\beta", from=1-2, to=2-3]
	\arrow["p", dashed, from=2-1, to=2-3]
	\arrow["f"', from=2-1, to=3-2]
	\arrow["g", from=2-3, to=3-2]
\end{tikzcd}\]

We can write 
    \[ K_W\sim_{\mathbb{Q}} \alpha^*(K_X+\Delta) + E.  \]
    \[ K_W\sim_{\mathbb{Q}} \beta^*(K_Z+\Delta') +F.\]
Since $p^{-1}$ has no exceptional divisor, hence any $\alpha$-exceptional divisor is also $\beta$-exceptional. Therefore $E-F$ is $\beta$-exceptional. Notice that any vertical curve for $\beta$ is $\alpha^*(K_X+\Delta)$-trivial by (3), consequently,  $E-F\equiv_\beta 0$. Similarly, $F-E\equiv_\beta 0$  also holds. Therefore by the negativity lemma (see \cite[Lemma 1.3]{Wang21}) we get $E=F$. Then $\alpha^*(K_X+\Delta)\sim_{\mathbb{Q}}\beta^*(K_Z+\Delta')$. Hence, $p$ is $B$-bimeromorphic.
    \end{proof}
\end{lemma}
\begin{lemma}[Rigidity Lemma]\cite{BS95}\label{rig}
   Let $f:X \to Y$ be a proper holomorphic morphism with connected fibers between two normal analytic varieties. Let $g$ be a holomorphic morphism from $X$ to an analytic space $Z$ such that $g(F)$ is a point for some fiber $F$ of $f$. Then there exists an open set $U$ of $Y$ with $F\subset f^{-1}(U)$ and a holomorphic morphism $h:U \to Z$ such that $g|_{f^{-1}(U)}=h \circ f|_{f^{-1}(U)}$. If $f$ has fibers with equal dimension, i.e., $f$ is flat and $g$ is proper, then $U$ can be taken to be $Y$.
\end{lemma}
\begin{lemma}\label{semi}
    Let $f:X \to Y $ be a surjective morphism with connected fibers between two compact analytic varieties. Assume that $X$ is normal and $Y$ is seminormal. Then $f_*\mathcal{O}_X=\mathcal{O}_Y$.
\end{lemma}
\begin{proof}
    By Stein factorization, we can factor $f$ as
    
\[\begin{tikzcd}
	X & Z & Y
	\arrow["g", from=1-1, to=1-2]
	\arrow["\mu", from=1-2, to=1-3]
\end{tikzcd}\]
    where $g_*\mathcal{O}_X=\mathcal{O}_Z$ and $\mu:Z \to Y$ is a finite morphism.
    Since $f$ has connected fibers, and $g$ has also connected fibers by construction, it follows that $\mu$ is bijective. Since $\mu$ is a proper bijective  morphism, hit is a homeomorphism.  Moreover, $\mu$ induces isomorphism on residue fields. Finally, since $Y$ is seminormal, any finite bijective morphism that is an isomorphism on residue fields must be an isomorphism. Therefore, $\mu$ is an isomorphism. Hence $f_*\mathcal{O}_X=\mathcal{O}_Y$.
\end{proof}
The following lemma is a complex analytic version of Ehresmann's fibration theorem.
\begin{lemma}\cite{PP25,FG65}\label{biholo}
    Let $f:X \to Y$ be a proper submersion of complex manifolds. Suppose that the fibers are analytically isomorphic to a fixed complex manifold $F$. Then, $f$ is locally trivial, i.e., for every $s \in Y$, there exists an open neighbourhood $U$ of $s$ such that $f^{-1}(U)$ and $U\times F$ are biholomorphic.
\end{lemma}

\begin{corollary}
    Let $f:X \to Y$ be a surjective morphism between normal compact  analytic varieties such that the general fibers are isomorphic to a fixed compact complex manifold $F$. Then there exists an analytic  open subset $U \subset Y$ such that $f^{-1}(U)$ is biholomorphic  to $ U\times F$.
\end{corollary}
\begin{proof}
    Let $Z:=f(\Sing(X))\cup \Sing(Y)$ and $W:=Y\setminus Z$. Then $f^{-1}(W) \to W$ is a proper holomorphic map between complex manifolds. By  \cite[Proposition 1.21]{PR96}, we can replace $W$ by a smaller open subset over which  $f$ is a submersion. Now
     choose a Zariski open  set $W'$ such that the fiber over each point  is isomorphic to $F$. Then replace $W$ by $W \cap W'$, we get our required result by applying Lemma \ref{biholo} for $X=f^{-1}(W)$ and $Y=W$.
\end{proof}

\begin{theorem}[Classification Theorem]\label{thm:classification}
    Let $X$ be a compact K\"ahler manifold of dimension $2$. Assume that $X$ is minimal, i.e., it does not contain any $(-1)$-curve, and Kodaira dimension $\kappa(X, K_X)=0$. Then $X$ is one of the following:
\begin{enumerate}
    \item $X$ is an Enriques Surface. In this case $X$ is projective.
    \item $X$ is a bi-elliptic surface. In this case $X$ is projective.
    \item $X$ is a $K3$ surface.
    \item $X$ is a torus.
\end{enumerate}
    
\end{theorem}

\begin{proof}
    Since $\kappa(X, K_X)=0$, from Table $10$ of \cite[Chapter VI, page 244]{BHPV04} we see that $X$ is one of the following: (i) Enriques surface, (ii) Bi-elliptic surface, (iii) Primary Kodaira surface, (iv) Secondary Kodaira surface, (v) $K3$ surface, or (vi) Abelian surface. From the same table we see that in case of (iii) and (iv), the first betti number $b_1(X)$ is odd, hence they are not K\"ahler. Therefore $X$ cannot be of type (iii) or (iv) (as it is K\"ahler). The only thing that remains to show is that in case (i) and (ii), $X$ is projective. From the Table $10$ again we notice that in case of (i) and (ii), the algebraic dimension $a(X)=2=\dim X$, and hence, $X$ is Moishezon. Since $X$ is also K\"ahler, from Moishezon's theorem it follows that $X$ is projective. 
\end{proof}

\begin{lemma}\label{abelian}
    Let $X$ be a $2$-dimensional torus equipped with an elliptic fibration $\phi:X \to C$ where $C$ is a curve. Then general fibers of $\phi$ are isomorphic.
\end{lemma}
\begin{proof}
Let $F$ be a general fiber of $\phi$. Then $F$ admits a closed embedding into $X$. By \cite[Proposition 1.2.1]{BL04}, we may assume that $F$ is  a subgroup of $X$. Let $q: X \to X/F$ be the quotient  morphism of group varieties. Note that dim$X/F=1$. Let $F'$ be any other general fiber of $\phi$. Since $\Supp F \cap \Supp F' =\phi$, we have $q(F')=${pt} which implies that  $F'$ is a translate of $F$. In particular, any two general fibers are translates of each other and hence, they are isomorphic.\\
\end{proof}
\begin{lemma}\label{action}
    Let $G$ be an infinite group acting on a set $X$. If the orbit of every element of $X$ is finite then, there exists a subgroup $H$ of $G$ of finite index which fixes three elements of $X$.
\end{lemma}
\begin{proof}
  Consider the induced action of $G$  on the product $X \times X \times X$. For $x\in X$, let $\Orb(x)$ denote the orbit of $x$ under the action of $G$ on $X$. Then for any $(x,y,z)\in X \times X \times X$ , we have $\Orb(x,y,z) \subset \Orb(x) \times \Orb(y)\times \Orb(z)$,  where $\Orb(x,y,z)$ denotes the orbit of $(x,y,z)$ under the action of $G$ on $X \times X \times X$ . In particular, $\Orb(x,y,z)$ is finite. By Orbit-Stabilizer theorem, $\Stab(x,y,z)$ is a subgroup of $G$ of finite index which fixes $x,y$  and $z$.
\end{proof} 
\begin{lemma}\label{lem:Canonical bundle}
    Let $ f:X \to Z$ be a morphism with connected fibers between normal compact K\"ahler  varieties. Let $\Delta\geq 0$ be a $\mathbb{Q}$- divisor such that $(X,\Delta)$ is lc. Assume that $K_X+\Delta \sim_{\mathbb{Q}} f^* D$ for some $\mathbb Q$-divisor $D$ on $Z$, and  $K_Z$  is represented by a Weil divisor. Let $S$ be an irreducible component of $\rddelta$ such that the restriction morphism $f|_{S}:S \to Z$ is generically finite. Then there exists an effective $\mathbb Q$-divisor $B_Z$ on $Z$ such that $(Z,B_Z)$ is lc and $K_X+\Delta \sim_{\mathbb{Q}} f^*(K_Z+B_Z)$.
\end{lemma} 
\begin{proof}
First replacing $(X,\Delta)$ by a dlt model as in \cite[Lemma 2.28]{DH25}, we may assume that $(X,\Delta)$ is a $\mathbb{Q}$-factorial dlt pair. Then $S$ is normal (being a dlt center), and by adjunction we have $(K_X+\Delta)|_S\sim_\mathbb{Q}K_S+\Delta_S$ such that $(S,\Delta_S)$ is dlt. Clearly, we have $K_S+\Delta_S\sim_{\mathbb{Q}}(f|_S)^*D$. Then by \cite[Lemma 1.1]{FG12} (which holds equally well for analytic varieties), there is an effective $\mathbb{Q}$-divisor $B_Z\geq 0$ such that $(Z,B_Z)$ is lc and $K_S+\Delta_S \sim_{\mathbb{Q}}(f|_S)^*(K_Z+B_Z)$. Now recall that $K_Z$ is a Weil divisor, and since $f|_S$ is generically finite, it follows that $K_S$ is also a Weil divisor. Therefore applying $(f|_S)_*$ as a cycle pushforward on both sides of $(f|_S)^*D \sim_{\mathbb{Q}}(f|_S)^*(K_Z+B_Z)$, we get that $D \sim_{\mathbb{Q}}(K_Z+B_Z)$. Substituting back to original relation, we conclude that $K_X+\Delta \sim_{\mathbb{Q}}f^*(K_Z+B_Z)$.

\end{proof}
\begin{lemma}\label{lem:divisor}
    Let $X$ be a $\mathbb{Q}$-factorial compact K\"ahler surface. Then $K_X$ is represented by a  divisor.
\end{lemma}
\begin{proof}

Let $f: Y \to X$ be a resolution. Then we can write $K_Y+B=f^*K_X$. Now we can run $K_Y$-MMP. There are two possible outcomes:
\begin{enumerate}
    \item We obtain a morphism $g: Y \to Y'$ a morphism such that $K_{Y'}$ is nef. Then $K_{Y'}$ is semiample. Therefore, $K_{Y'}$ is represented by a divisor. This implies that $K_Y$ is also represented by a divisor.\\
 \item We obtain a Mori fiber space $g:Y \to Z$. Because $Y$ is a surface, we have $\dim Z\leq 1$. In this situation, $Y$ is projective. Hence $K_Y$ is represented by a divisor. 
 \end{enumerate}
 In both cases $K_Y$ is represented by a divisor. Since $K_Y+B=f^*K_X$, it follows that $K_X$ is also represented by a divisor.
\end{proof}

\section{Finiteness of Pluricanonical Representation}\label{sec:finiteness}
In this section we prove the finiteness of the pluricanonical representation for compact K\"ahler surfaces of positive (log) Kodaira dimension. Fujino's proof of the same result in \cite[Theorem 3.5]{Fu00} relies on a classical finiteness theorem for Moishezon manifolds due to Ueno (see \cite[Theorem 14.10]{KU75}), which is unavailable to us. Our approach is very different. We use the classification of compact K\"ahler surfaces with a detailed case-by-case analysis, see Theorem \ref{S}, Theorem \ref{plu}.

We start with the definition of pluricanonical representation and prove some elementary properties of it.
\begin{definition}\label{def:b-bimeromorphic}
    Let $(X,\Delta_X)$ and $(Y,\Delta_Y)$ be two log pairs of normal compact analytic varieties. Let $f:X \dashrightarrow Y$ be a bimeromorphic map. We say that  $f:(X,\Delta_X) \dashrightarrow (Y,\Delta_Y)$ is $B$-bimeromorphic if there exists a common resolution $p:Z \to X, q:Z \to Y$ such that $K_Z+\Delta_Z \sim _{\mathbb{Q}}p^*(K_X+\Delta_X)\sim _{\mathbb{Q}} q^*(K_Y+\Delta_Y) $. \\
    Now we can extend this definition for pure $n$-dimensional normal compact reduced analytic space as follows:\\
    Let $(Y,\Gamma):= \sqcup_{i=1}^{k}(Y_i,\Gamma_i)$ and $(Y',\Gamma'):= \sqcup_{i=1}^{k}(Y'_i,\Gamma'_i)$ be two such pairs.\\
    We say $f:(Y,\Gamma) \dashrightarrow (Y',\Gamma') $ is a $B$- bimeromorphic map (resp. morphism) if $Y \dashrightarrow Y'$ is a proper bimeromorphic map (resp. morphism) and there exists a common resolution $\alpha: Z \to Y,\beta: Z \to Y'$ of $f$ such that $\alpha^*(K_Y+\Gamma)\sim_{\mathbb{Q}}\beta^*(K_{Y'}+\Gamma')$ . That is, there exists a permutation $\sigma$ such that $f_i:Y_i \dashrightarrow Y'_{\sigma(i)}$ is a proper bimeromorphic map (resp. morphism) and a common resolution $\alpha_i:Z_i \to Y_i, \beta_i:Z_i \to Y'_{\sigma(i)}$ of $f_i$ such that $\alpha^*_i(K_{Y_i}+\Gamma_i)\sim _{\mathbb{Q}} \beta_i^*(K_{Y'_\sigma(i)}+\Gamma'_{\sigma(i)})$  for all $i\in\{1,2,\ldots, k\}$.\\
  We define
  \[\Bim(X,\Delta):=\{g : (X,\Delta) \dashrightarrow (X,\Delta)\;|\; g \mbox{ is } B\mbox{-bimeromorphic}\} \]
    \[\Aut(X,\Delta):=\{g: X \to X \;|\; g \mbox{ is automorphism and } g^*\Delta=\Delta\}.\]

Note that $\Bim(X,\Delta)$ and $\Aut(X,\Delta)$ have natural group structures under composition of maps.
Since Bim($X,\Delta)$ acts on $H^0(X,\mathcal{O}_X(m(K_X+\Delta)))$ for every integer $m$ such that $m(K_X+\Delta)$ is Cartier, we can define $B$- pluricanonical representation $\rho_m :\Bim(X,\Delta) \to \Aut H^0(X,m(K_X+\Delta))$.\\

\begin{remark}\label{rmk:b-bimeromorphic}
We note that if $f:(X,\Delta)\dashrightarrow (Y,\Gamma) $ is a $B$-bimeromorphic map, then it induces isomorphism $\Bim(X,\Delta) \cong  \Bim(Y,\Gamma)$ via $\sigma \to f\circ \sigma \circ f^{-1}$.
\end{remark}

\end{definition}
\begin{lemma}\label{terminal}
    Let $(X,\Delta)$ be a terminal pair of dimension  $2$. Then $\Bim(X,\Delta)\cong \Aut(X,\Delta)$.
    \begin{proof}
        Let $\sigma: (X,\Delta)\dashrightarrow (X,\Delta)$ be a $B$-bimeromorphic map. Let $W$ be a common resolution satisfying the following commutative diagram:

\[\begin{tikzcd}
	& W \\
	{(X,\Delta)} && {(X,\Delta)}
	\arrow["\alpha"', from=1-2, to=2-1]
	\arrow["\beta", from=1-2, to=2-3]
	\arrow["\sigma", dashed, from=2-1, to=2-3]
\end{tikzcd}\]

such that $\alpha^*(K_X+\Delta)\sim_{\mathbb{Q}} \beta^*(K_X+\Delta)$.
We may write log equations for the maps $\alpha$ and $\beta$.
\[K_W+(\alpha^{-1})_*\Delta\sim_{\mathbb{Q}}\alpha^*(K_X+\Delta) +E\]
\[ K_W +(\beta^{-1})_*\Delta\sim_{\mathbb{Q}}\beta^*(K_X+\Delta)+F\]
Let $D$ be a $\sigma$-exceptional divisor. Let $D'$ be strict transform of $D$ under the morphism $\alpha$ on $W$. Then $D'$ is also $\beta$-exceptional. Since $(X,\Delta)$ is a terminal pair, discrepancy $a(D',X,\Delta)>0$. Again since $D$ is a divisor on $X$, the discrepancy $a(D',X,\Delta)\leq 0$. This is a contradiction. Hence $\Ex(\sigma)$ has codimension at least $2$ in $X$. Similarly, we can show $\Ex(\sigma^{-1})$ has codimension at least $2$. In particular, $\alpha$ and $\beta$ exceptional divisors coincide. Since $X$ is a surface, $\sigma$ is an isomorphism. Moreover, $\sigma_*\Delta=\Delta$  and hence $\sigma \in \Aut(X,\Delta)$.
    \end{proof}
\end{lemma}

Let $h:X \to Y$ be a fibration between normal analytic varieties. Let $\Delta$ be an effective $\mathbb{Q}$-divisor  on $X$.\\
For every prime divisor $P\subset Y$, define $B(h;\Delta):=\sum_{P\subset Y} (1-\lct(X,\Delta,h^*P))P$ where $\lct(X,\Delta,f^*P)=\sup\{t\; |\; K_X+\Delta+th^*P  \mbox{ is lc }\}$.
\begin{theorem}\label{thm:Fujita Canonical}
    Let $f:X \to C$ be a relatively minimal elliptic surface. Let $\Delta\geq 0$ be an effective $\mathbb{Q}$-divisor such that $(X,\Delta)$ is klt. Assume that $K_X+\Delta \sim_{\mathbb{Q}}f^*L$ for some line bundle $L$ on $C$. Then we can write $K_X+\Delta \sim_{\mathbb{Q}}f^*(K_C+B+M)$ such that $B=B(f;\Delta)$ and the pair $(C,B+M)$ is also klt.
\end{theorem}

\begin{proof}
    By Fujita's canonical bundle formula (see \cite[8.2.1]{Cor07}), we have
\begin{equation}
    K_X \sim_{\mathbb{Q}}f^*(K_C+B(f;0)+M)
\end{equation}
where $M \sim _{\mathbb{Q}}\frac{1}{12}j^*\mathcal{O}(1)$ and $j:C \to \mathbb{P}^1$ is the $j$-invariant function of smooth fibers of $f$.\\
Since $\Delta=f^*D$, we have $B(f;\Delta)=D+B(f;0)$ by \cite[Lemma 7.4(ii)]{SP09}. Hence we have,
\begin{equation}
    K_X+\Delta\sim _{\mathbb{Q}}f^*(K_C+B(f;\Delta)+M).
\end{equation}
Note that since $(X,\Delta)$ is klt, every coefficient of $B(f;\Delta)$ is strictly less than $1$. Since $C$ is smooth $(X,B(f;\Delta))$ is also klt. Observe that $M$ is effective. If $M=0$, then we are done. If $M\neq 0$, then it is a non zero ample divisor. Choose a sufficiently large $n\in \mathbb{N}$ such that $nM \sim S$ where $S$ is  a reduced divisor. Then $M \sim_{Q} \frac{1}{n} S$. We can choose $n$ to be very large such that $B(f;\Delta)+\frac{1}{n}S)$ is still klt. Taking $J=\frac{1}{n}S$, we are done.
\end{proof}

\begin{lemma}\label{lem: Image}
    Let $X$ be a compact K\"{a}hler manifold such that $(X,\Delta)$ is a lc pair. Assume that $K_X+\Delta$ is semiample. Let $Y=\Proj \bigoplus_{d\geq 0} H^0(dm(K_X+\Delta))$ and let $\phi:X \to Y$ be the morphism obtained from Iitaka fibration. The morphism $\phi$ induces a homomorphism $\chi: \Aut(X,\Delta) \to \Aut Y$ . Then the image of $\chi$ is contained in $\Aut(Y,B)$, where $K_X+\Delta \sim_{\mathbb{Q}}\phi^*(K_Y+B+J)$ and $B$ and $J$ are respectively the discriminant part and the moduli part .
\end{lemma}
\begin{proof}
     For every prime divisor $P\subset Y$, the coefficient of $P$ in $B$ is given by $1-\operatorname{lct}(X,\Delta; \phi^*P)$. Let $g\in \Aut(X,\Delta)$ and set $g':=\chi(g)$. Then we have the following commutative diagram:

\[\begin{tikzcd}
	X && X \\
	C && C
	\arrow["g", from=1-1, to=1-3]
	\arrow["\phi", from=1-1, to=2-1]
	\arrow["\phi"', from=1-3, to=2-3]
	\arrow["{g'}"', from=2-1, to=2-3]
\end{tikzcd}\]
Now we have,
    \begin{align*}
\operatorname{mult}_P(B)
  &= 1 - \operatorname{lct}(X, \Delta; \phi^*P) \\
  &= 1 - \operatorname{lct}(X,\Delta; g^*\phi^*P) \ \ [g^*(K_X+\Delta)=(K_X+\Delta)] \\
  &= 1 - \operatorname{lct}(X, \Delta; \phi^*g'^*P) \\
  &= 1 - \operatorname{lct}(X, \Delta; \phi^*Q) \ \ [g'^*P=Q] \\
  &= \operatorname{mult}_Q(B).
\end{align*}
This shows that $g'^*B=B$ and hence, $g'\in \Aut(Y,B)$.
\end{proof}
\begin{theorem}\cite[Theorem 3.3]{Fu00}\label{C}
    Let $(C,\Delta)$ be a compact K\"{a}hler lc curve. Then there is a positive integer $m$ such that $\rho_{km}(\Bim(C,\Delta))$ (see Definition \ref{def:b-bimeromorphic}) is finite for every $k \in \mathbb{N}$.
\end{theorem}
\begin{proof}
    This follows from \cite[Theorem 3.3]{Fu00} by simply observing that any compact complex analytic space of dimension $1$ is projective, for example, see \cite[Lemma 1.20 and Remark 1.21]{DO24}.
\end{proof}

\begin{theorem}\label{S}
    Let $(X,\Delta)$ be a compact K\"ahler klt pair of dimension $2$ such that $K_X+\Delta$ is nef with Kodaira dimension $\kappa(X,K_X+\Delta)=1$. Then there is a positive integer $m$ such that $\rho_{km}(\Bim(X,\Delta))$ is finite for every $k \in \mathbb{N}$.
\end{theorem}
\begin{proof}
    Since $(X,\Delta)$ is a klt pair, we have a terminal model, say $f:(X',\Delta')\to (X,\Delta)$ such that $K_{X'}+\Delta'\sim_{\mathbb{Q}}f^*(K_X+\Delta)$ and $(X', \Delta')$ has $\mathbb Q$-factorial terminal singularities. Then by Lemma \ref{terminal}, $\Bim(X',\Delta')\cong \Aut(X',\Delta')$. Since $\Delta'\geq 0$ is effective, $X$ has also terminal singularities, and hence, $X$ is smooth. Replacing $(X,\Delta)$ by $(X',\Delta')$ and $\Bim(X,\Delta)$ by $\Aut(X,\Delta)$, we may assume that  $X$ is a compact K\"ahler manifold of dimension $2$. Since $K_X+\Delta$ is nef and  $\kappa(X, K_X+\Delta)=1$, by Theorem \ref{abun}, $K_X+\Delta$ is semiample, and hence  we can define a morphism with connected fibers $\phi:X \to C$ to a smooth projective curve $C$ such that $K_X+\Delta\sim_{\mathbb Q} \phi^* L$ for some ample $\mathbb Q$-divisor $L$ on $C$. Now we run a $K_X$-relative MMP over $C$, this contracts all $(-1)$-curves in the fibers of $\phi$. Then replacing $X$ by the output of this MMP and using Remark \ref{rmk:b-bimeromorphic}, we may assume that $\phi:X\to C$ is a relatively minimal fibration, i.e., $K_X$ is $\phi$-nef.
 Recall that $K_X+\Delta\equiv_\phi 0$ and $\Delta$ is effective, and thus for a general fiber $F$ of $\phi$, $\mbox{deg}(K_F)=K_X\cdot F= -\Delta\cdot F\leq 0$. This implies that the general fibers of $\phi$ are either $\mathbb{P}^1$ or elliptic curves, i.e., $\phi$ is either a $\mathbb{P}^1$-bundle over $C$ or a minimal elliptic fibration over $C$.\\

 Assume that $X$ is not a minimal surface. Since $\phi:X\to C$ is relatively minimal fibration, there is no $(-1)$-curve in the fibers of $\phi$.  On the other hand, since $X$ is not minimal, there must be a $(-1)$-curve, say $D$ on $X$. Then $D$ dominates $C$, as it is not contained in any fiber of $\phi$. Then  by \cite[Corollary 2.7]{Lin22}, $X$ is projective. Therefore finiteness of pluricanonical representation follows from \cite[Theorem 3.4]{Fu00}.  So from now we will assume that $X$ is a minimal surface.\\

Now we will use Enriques–Kodaira classification of minimal surfaces and show that in each case $\rho_{km}(\Bim(X,\Delta))$ is finite for every $k\in\mathbb N$. We will split our discussion into three main cases based on the Kodaira dimension of $X$.\\

\noindent
\textbf{Case I}: $\kappa(X, K_X)= -\infty$.\\
In this case $H^2(X, \mathcal O_X)=H^0(X, K_X)^*=0$, and thus from the proof of \cite[Lemma 2.43]{DH25} it follows that $X$ is projective. Hence, by \cite[Theorem 3.4]{Fu00}, $\rho_{km}(\Bim(X,\Delta))$ is finite for all $k\in\mathbb N$.\\

\noindent
\textbf{Case II} : $\kappa(X, K_X)=0$.\\
In this case by Theorem \ref{thm:classification}, we know that $X$ is one of the following:
\begin{enumerate}
    \item $X$ is a projective Enriques Surface. 
    \item $X$ is a projective bi-elliptic surface.
    \item $X$ is a $K3$ surface.
    \item $X$ is a torus.
\end{enumerate}
Since $X$ is projective in case (1) and (2), the finiteness of $\rho_{km}(\Bim(X,\Delta))$  follows from \cite[Theorem 3.4]{Fu00}.

So from now on we will assume that $X$ is either a $K3$ surface or a $2$-dimensional torus with an elliptic fibration $\phi:X\to C$. Choose $m\in \mathbb{N}$ such that $m(K_X+\Delta) \sim \phi^*\mathcal{O}_C(1)$. Then $C=\mbox{Proj}\bigoplus_{d\geq 0} H^0(dm(K_X+\Delta))$. Let $\chi:\Aut(X,\Delta)\to \Aut C$ be the morphism induced by $\phi$. Then we have the following commutative diagram:
\begin{equation}\label{diag:aut}
\begin{tikzcd}
	{\Aut(X,\Delta)} && {\Aut_\mathbb{C}(H^0(m(K_X+\Delta))} \\
	& {\Aut C}
	\arrow[ "\rho_m", from=1-1, to=1-3]
	\arrow[ "\chi", from=1-1, to=2-2]
	\arrow[from=1-3, to=2-2]
\end{tikzcd}
\end{equation}
From the diagram above we get the following exact sequence of groups
\begin{equation}\label{eqn:aut}
    1 \to K\to \rho_m(\Aut(X,\Delta))\to \chi(\Aut(X,\Delta))\to 1.
\end{equation}

To show that $\rho_m(\Aut(X,\Delta))$ is finite, it is enough to show that both $\chi(\Aut(X,\Delta))$ and $K$ are finite.
First we will show that $\chi(\Aut(X,\Delta))$ is finite. We will consider the torus and $K3$ surface cases separately. 

First assume that $X$ is a $2$-dimensional torus. Now by Theorem \ref{thm:Fujita Canonical}, we have
     \begin{equation}\label{eqn:cbf}
         K_X +\Delta \sim_{\mathbb{Q}} \phi^*(K_C+B_C+M_C),
     \end{equation}
     where $B_C$ is the discriminant divisor and $M_C$ the moduli part. The moduli divisor $M_C$ is a $\mathbb{Q}$-divisor corresponding to the ample $\mathbb{Q}$-line bundle $\frac{1}{12}j^*\mathcal{O}(1)$, where $j:C \to \mathbb{P}^1$ is the $j$-invariant function of the smooth elliptic fibers. Now by Lemma \ref{abelian}, the general fibers of $\phi: X \to C$ are isomorphic elliptic curves, hence $j:C \to \mathbb{P}^1$ is the constant morphism; in particular,  $M_C=0$. Therefore we have
     \[K_X +\Delta \sim_{\mathbb{Q}} \phi^*(K_C+B_C)\]
Now observe that $(C, B_C)$ is klt by Theorem \ref{thm:Fujita Canonical} and $K_C+B_C\sim_{\mathbb{Q}}\mathcal{O}_C(1)$ is ample. This implies that $\Aut(C,B_C)$ is finite by \cite[Proposition 3.4]{LS25}, and hence, $\chi(\Aut(X,\Delta))$ is finite by Lemma \ref{lem: Image}.\\

Now we will consider the other non-projective case, namely, $\phi:X\to C$ is an elliptic $K3$ surface. Then $C \cong \mathbb{P}^1$. [Note that if $g(C)\geq 1$, then $C$ admits a nonzero global holomorphic 1-form. Pulling it back via $\phi$, yields a non zero global holomorphic $1$ form on $X$. However, since $X$ is a $K3$ surface,we have $H^0(X,\Omega_X^1)=0$ (equivalently, $H^1(X,\mathcal{O}_X)=0$), which is a contradiction.]\\
If the image of $\chi$  in diagram \eqref{diag:aut} is finite, then we are done. So we assume that $\mbox{im}(\chi)=G$ is an infinite group. Note that $G$ acts faithfully on $\mathbb{P}^1$. If the cardinality of the orbit $\Orb(p)$ is finite for every $p\in \mathbb{P}^1 $, then by Lemma \ref{action} there exists a subgroup $H\subset G$ of finite index that fixes three points on $\mathbb{P}^1$; in particular, $H$ fixes $\mathbb{P}^1$. On the other hand, since $G$ is infinite, $H$ must be non-trivial, this contradicts the faithfulness of the action. Hence, there exists a $p_0\in \mathbb{P}^1$ such that $\Orb(p_0)$ contains infinitely many points. Then the fibers of $\phi$ over the points of $\Orb(p_0)$ are isomorphic to each other. Since $\phi$ is an elliptic fibration, this implies that the $j$- invariant morphism from $C=\mathbb{P}^1 \to \mathbb{P}^1$  which is a constant. Then by a similar argument as in Case II, it follows that the moduli part $M_C=0$ in the canonical bundle formula \eqref{eqn:cbf}. We also have $(C, B_C)$ is klt, and $K_C+B_C$ is ample, and consequently $\Aut(C, B_C)$ is finite by \cite[Proposition 3.4]{LS25}; hence, $G\subset \Aut(C, B_C)$ is finite, which is a contradiction.\\

Now we will show that $K$ is finite.\\
Note that for every $g\in K$ there exists a $\lambda\in\mathbb{C}^\times$ such that $g^*s=\lambda s$ for all $s\in H^0(X, m(K_X+\Delta))$. Hence we can identify $K$ as a subgroup of $\mathbb{C}^{\times}$.
To show $K$ is finite, it is enough to show there exists a natural number $N$ such that $g^N=1$ for all $g\in K$. We now proceed to establish this. Since $g\in K$, it follows that $\chi(g)=\mbox{id}$. Hence $g$ induces an automorphism on every fiber of $\phi$. Fix a general fiber $F \subset X$.
    Consider the following commutative diagram :
  
\[\begin{tikzcd}
	F && X \\
	F && X
	\arrow["i", hook, from=1-1, to=1-3]
	\arrow["{g|_F}"', from=1-1, to=2-1]
	\arrow["g", from=1-3, to=2-3]
	\arrow["i", hook, from=2-1, to=2-3]
\end{tikzcd}\]
which  induces the following  exact sequence:
\begin{equation}
    0 \to \mathcal{O}_X(m(K_X+\Delta)-F) \to \mathcal{O}_X(m(K_X+\Delta))\to \mathcal{O}_F(m(K_F+\Delta_F))\cong \mathcal{O}_F \to 0
\end{equation}
Since $m(K_X+\Delta)$ is globally generated, we may choose a section $s \in H^0(X,m(K_X+\Delta))$ such that $s|_F \neq 0$. As $H^0(F, m(K_F+\Delta_F))\cong H^0(F,\mathcal{O}_F) \cong \mathbb{C}$, we obtain the following diagram  :

\[\begin{tikzcd}
	0 & {H^0(m(K_X+\Delta)-F)} & {H^0(m(K_X+\Delta))} & {H^0(m(K_F+\Delta_F))} & 0 \\
	0 & {H^0(m(K_X+\Delta)-F)} & {H^0(m(K_X+\Delta))} & {H^0(m(K_F+\Delta_F))} & 0
	\arrow[from=1-1, to=1-2]
	\arrow[from=1-2, to=1-3]
	\arrow[from=1-3, to=1-4]
	\arrow["{g^*}", from=1-3, to=2-3]
	\arrow[from=1-4, to=1-5]
	\arrow["{(g|_F)^*}", from=1-4, to=2-4]
	\arrow[from=2-1, to=2-2]
	\arrow[from=2-2, to=2-3]
	\arrow[from=2-3, to=2-4]
	\arrow[from=2-4, to=2-5]
\end{tikzcd}\]

Since $(X,\Delta)$ is klt, we may choose $F$ such that $(F,\Delta_F)$ is also klt. Note that $(g|_F) \in \Aut(F,\Delta_F)$ and this also acts by multiplication of $\lambda$. Applying \cite[Proposition 4.9]{YG13} for the pair $(F,\Delta_F)$, we obtain $\lambda^N=1$. Since this number does not depend on $g$,
we get our desired result.\\

\textbf{Case III}: $\kappa(K_X)=1$\\
Recall the following exact sequence \eqref{eqn:aut}
\[1 \to K\to \rho_m(\Aut(X,\Delta))\to \chi(\Aut(X,\Delta))\to 1.\]
Since $X$ is a minimal surface with $\kappa(X, K_X)=1$, $K_X$ is nef, and hence it is semiample by Theorem \ref{abun}. Let $\psi:X \to Y$ be the semiample fibration to a curve $Y$. Now observe that $H^0(X, mK_X) \subset H^0(X, m(K_X+\Delta))$ for all sufficiently divisible $m\in \mathbb{N}$, this implies that $\psi:X\to Y$ factors through $\phi$. However, since $\phi$ and $\psi$ both have connected fibers, $C\to Y$ is an isomorphism (as both of them are curves).

Now consider the following commutative diagram:
\[\begin{tikzcd}
	{\Aut(X,\Delta)} && {\Aut(C)} \\
	{\Aut(X)} && {\Aut(Y)}
	\arrow["\chi", from=1-1, to=1-3]
	\arrow[hook, from=1-1, to=2-1]
	\arrow["\cong", from=1-3, to=2-3]
	\arrow["\eta"', from=2-1, to=2-3]
\end{tikzcd}\]
By \cite[Proposition 1.2]{PS20}, we know that $\mbox{im}(\eta)$ is finite. Thus, from the diagram above, it follows that $\mbox{im}(\chi)$ is finite. 
Then again by a similar argument  as in Case II above, it follows that the kernel $K$ in \eqref{eqn:aut} is finite. This completes the proof of our theorem. 
\end{proof}

The following result shows that the finiteness of the pluricanonical representation holds for log canonical compact K\"ahler surfaces of positive (log) Kodaira dimension. This is one of the fundamental technical tools of this article. 
\begin{theorem}\label{plu} 
 Let $(X,\Delta)$ be a compact K\"{a}hler lc surface such that $K_X+\Delta$ is nef. Assume that one of the following conditions holds:
 \begin{enumerate}
     \item the Kodaira dimension $\kappa(X, K_X+\Delta)>0$, or
     \item $\kappa(X, K_X+\Delta)=0$ and $\Delta\neq 0$.
 \end{enumerate}
 Then there is a positive integer $m$ such that $\rho_{km}(\Bim(X,\Delta))$ is finite for every $k \in \mathbb{N}$.
\end{theorem}
\begin{proof}
By  Theorem \ref{abun}, $K_X+\Delta$ is semiample. Let $\phi:X \to Y$ be the semiample fibration with connected fibers. We can find a dlt pair $(Z,\Gamma)$ such that $K_Z+\Gamma\sim_\mathbb{Q} K_X+\Delta$ by \cite[Lemma 2.28]{DH25}. We can replace $(X,\Delta)$ by $(Z,\Gamma)$ using Remark \ref{rmk:b-bimeromorphic}. From now on we may assume that $(X,\Delta)$ is a $\mathbb{Q}$-factorial dlt pair.

We will split our discussion into three main cases based on the Kodaira dimension of $K_X+\Delta$.\\

\textbf{Case I:} $\kappa(X,K_X+\Delta)=2$\\   
Since $K_X+\Delta$ is a big divisor, the variety $X$ is Moishezon. Moreover, as  $X$ is a K\"ahler variety with rational singularities, it follows from \cite[Theorem 1.6]{Nam02}, $X$ is projective. Then we are done by \cite[Theorem 3.5, Case 1]{Fu00}.\\

 The proof closely follows the argument of the proof of \cite[Theorem 3.5, Case 2]{Fu00}. The essential input is our Theorem \ref{S}, which is used at a crucial step. For the reader's convenience, we reproduce the argument here.\\
 
\textbf{Case II:} $\kappa(X,K_X+\Delta)=1$\\
    Let $h:X' \to X $ be a minimal resolution and $K_{X'}+\Delta'=h^*(K_{X}+\Delta)$ with $\Delta'$ i
    effective. We may assume that $X$ is a two-dimensional compact K\"{a}hler manifold. By Theorem \ref{abun}, we can define a morphism $\phi_{|m(K_X+\Delta)|}: X\to Y $ with connected fibers for a sufficiently large and divisible $m$.\\
   By using Lemma \ref{MMP} we can contract $(-1)$-curves in the fibers, and  we may assume that $X$ is minimal elliptic compact K\"{a}hler surface or a $\mathbb{P}^1$-bundle. If the horizontal component $\rddelta^h \neq 0$, we may take an irreducible component $S$ of $\rddelta^h$. Then we have the following exact sequence
\[0\to H^0(X, m(K_X+\Delta) -S)\to H^0(X,m(K_X+\Delta))\to H^0(S,m(K_X+\Delta)|_S).\]
Since $\phi(S)=Y$, we get that  $H^0(X,m(K_X+\Delta)- S)=0$. 
Then 
  \[  H^0(X,m(K_X+\Delta)) \subset H^0(S,m(K_X+\Delta)|_S). \]
  Since $S$ is a curve we are done by Theorem \ref{C}.\\
  Note that in the above situation, it follows from \cite[Corollary 2.7]{Lin22} that $X$ is projective. For completeness, we include the argument here.\\
  
  We may assume that $\rddelta^h=0$. When $\phi:X \to Y$ is a $\mathbb{P}^1$-bundle, we have $\lfloor \Delta \rfloor = \sum \phi^*p_i$ for some $p_i\in Y$. For elliptic case, we have $K_X \sim_{\mathbb{Q},\phi} 0 $ and $\Delta \equiv_\phi 0$. Then $\Delta=\sum a_i \phi^* p_i$ with $a_i>0$ (see \cite[VIII.3,VIII.4]{AB83}). Define $B:=\sum_{p_i \in \phi(\lfloor \Delta \rfloor)} p_i$. Since $Y$ is projective we can choose an ample Cartier divisor $H$ such that $k(K_X+\Delta)\sim \phi^*H$. Let $\mathcal{O}_Y(-B)$ be the ideal sheaf associated to the support of $B$.  $H$ being ample, $\mathcal{O}_Y(-B)\otimes \mathcal{O}_Y(tH)$ is globally generated for sufficiently large $t$. Then we have $tH \sim B+P$ for some effective divisor $P$. Then $tk(K_X+\Delta)\sim \phi^*(B+P)$. Take $A:=\phi^*B$. Then we get
\[ (l+t)(k(K_X+\Delta) -\frac{1}{l+t}A)\sim lk(K_X+\Delta)+\phi^*P, \]
and since $\phi^*P$ is effective we have 
\[ H^0(X,lk(K_X+\Delta)) \subset H^0(X,(l+t)k(K_X+\Delta - A')), \] where $A':=\frac{1}{(l+t)k}A$. Define $\Gamma=\Delta-A'$.
 $(X,\Gamma)$ be a klt pair. Note that $K_X+\Gamma \sim _{\mathbb{Q}}\phi^* L$. If $\kappa(K_X+\Gamma)=2$ then $X$ is projective. If $K_X+\Gamma$ is nef and $\kappa(K_X+\Gamma)=1$, then we are done by Theorem \ref{S}. If $K_X+\Gamma$ is not nef, then there exists a horizontal curve $C_1$ which dominates $C$ satisfying $(K_X+\Gamma)\cdot C_1 < 0$. Then by \cite[Corollary 2.7]{Lin22}, we get $X$ is projective. When $X$ is projective, we are done by \cite[Theorem 3.5, Case I]{Fu00}.\\
 
\textbf{Case III:} $\kappa(X,K_X+\Delta)=0$ and $\Delta\neq 0$ (cf. \cite[Theorem 3.5 Case 3]{Fu00}).\\  
In this case we can show Bim $(X,\Delta)$ acts trivially on $H^0(X,km(K_X+\Delta))$.  By Theorem \ref{imp} we can take a preadmissible section $s$. Then By Lemma \ref{restriction} $g^*(s|_{\lfloor \Delta \rfloor})=s|_{\lfloor \Delta \rfloor}$ for every $g\in $ Bim $(X,\Delta)$. Since $\kappa(K_X+\Delta)=0$, $H^0(X,l(K_X+\Delta))$ is $1$ dimensional for every $l\in \mathbb{N}$ and hence $g^*s=s$ for every $g$.  Hence, we have the result. Note that in this case $X$ is uniruled and hence $X$ is projective.
\end{proof}

\begin{remark}\label{K3}
  Note that there exists a non-algebraic $K3$ surface that admits a non-symplectic automorphism of infinite order acting on  holomorphic two forms via multiplication by a complex number which is not a root of unity. Consequently, $\rho(\Bim(X))$ is an infinite group (see \cite[ Example 2.3]{LS25}).
\end{remark}

\section{Admissible and Preadmissible sections}\label{sec:ad-and-pread}
In this section, we introduce two special classes of pluricanonical sections, called preadmissible and admissible sections, and show how they play a key role in proving abundance for semi log canonical pairs.\\
We begin our discussion with the following proposition. This is an analytic adaptation of Fujino's Proposition \cite[Proposition 2.1]{Fu00}.
\begin{proposition}\label{Pro}\cite[Proposition 2.1]{Fu00}
    Let $(X,\Delta)$ be a $\mathbb{Q}$-factorial compact  K\"{a}hler dlt pair of dimension  $n\leq 3$. Let $f:X \to Y$ be a proper surjective morphism to a normal projective variety with connected fibers. Assume that $K_X+\Delta\sim_{\mathbb{Q},f} 0$. Then one of the following holds.
   \begin{enumerate}
       \item dim $Y=0$.
       \subitem (1.1) $\lfloor \Delta \rfloor$ is connected.
       \subitem (1.2) $\lfloor \Delta \rfloor$ has two components $\Delta_1$ and $\Delta_2$, and there exists a meromorphic map $v: X \dashrightarrow (V,P)$ with general fiber $\mathbb{P}^1$. Moreover, there exists an irreducible component $D_i \subset  \Delta_i $ such that the restriction of $v$ on $D_i$ induces $B$-bimeromorphic map to $(V,P)$.
       \item  dim $Y\geq 1$.
       \subitem (2.1) $\lfloor \Delta \rfloor \cap f^{-1}(y)$ is connected for every $y\in Y$.
       \subitem (2.2) The number of connected components of $\lfloor \Delta \rfloor \cap f^{-1}(y)$ is at most $2$ for every $y\in Y$.
       There exists a meromorphic map $v:X \dashrightarrow (V,P)$ with general fiber $\mathbb{P}^1$. The horizontal part $\rddelta^h$ satisfies one of the following :\\
        (i) $\rddelta^h= D_1$, which is irreducible and has the deg $[D_1:V]=2$, there also exists a $B$-bimeromorphic involution on $D_1$ over $V$.\\
        (ii) $\rddelta^h =D_1+D_2$ such that $D_i$ is irreducible and restriction of $v$ on $D_i$ is a $B$-bimeromorphic map to $(V,P)$ for $i=1,2$.
   \end{enumerate} 
\end{proposition}
   Before proving the main proposition, we prove the following lemma.
   \begin{lemma}\cite[Lemma 2.3]{Fu00}\label{Fano}
       Let $(X,\Delta)$ be a $\mathbb{Q}$-factorial compact K\"{a}hler log canonical pair of dimension $n$ where $n=2$ or $3$, and $\lfloor \Delta \rfloor \neq 0$. Let $h:X \to Y$ be a proper surjective  morphism onto a normal compact analytic variety $Y$ with connected fibers. Assume that the following conditions hold:
       \begin{enumerate}
           \item $(X,\Delta-\epsilon \lfloor \Delta \rfloor)$ is klt for $\epsilon>0$ sufficiently small positive rational number.
           \item $K_X+\Delta\sim_{\mathbb{Q},h} 0$.
           \item There is a $K_X+\Delta -\epsilon \lfloor \Delta \rfloor$-extremal Fano contraction $u:X \to V$ over $Y$ such that $\dim V=n-1 $.
       \end{enumerate}
       Then the horizontal part $\rddelta^h$ of $\lfloor \Delta \rfloor$ is one of the following :
       \begin{enumerate}
           \item $\rddelta^h= D_1$ which is irreducible and $\deg[\rddelta^h:D_1]=2$
           \item $\rddelta^h=D_1$ which is irreducible and $ \deg[\rddelta^h:D_1]=1$
           \item $\rddelta^h=D_1+D_2$ where $D_i$ is irreducible and $\deg[D_i:V]=1$ for $i=1,2$.
       \end{enumerate}
       In case of  $(1)$ and $ (3)$, the number of connected components of $\lfloor \Delta \rfloor \cap h^{-1}(y)$ is at most $2$ for every $y\in Y$.\\
       In case $(2),  \lfloor \Delta \rfloor \cap h^{-1}(y)$ is connected for every $y\in Y$.\\
       Moreover, there is a $\mathbb{Q}$-divisor $P$ on $V$ such that $(V,P)$ is $\mathbb{Q}$-factorial log canonical pair of dimension $2$ and $(K_X+\Delta)|_{D_i} =u|_{D_i}^*(K_V+P)$ for $i=1,2$.\\
       In case $(1)$, there is a $B$-bimeromorphic involution $\iota:(D_1,\Diff(\Delta-D_1)) \dashrightarrow (D_1,\Diff(\Delta-D_1))$
over $V$.
In case $(3), u|_{D_i}:D_i \to V$  is a $B$-bimeromorphic morphism for $i=1,2$.

 \begin{proof}
    Let $\rddelta^h$ and $\rddelta^v$ be the horizontal and vertical parts of $\lfloor \Delta \rfloor$ with respect to the morphism $u$, respectively . For a general fiber $F$ of $u$, we have $\rddelta^v\cdot F=0$. Since $\lfloor \Delta \rfloor$ is relatively $u$-ample, we get $\rddelta^h\neq 0$. Again, the general fiber of $u:X \to V$ is $\mathbb{P}^1$. Since $K_X+\Delta$ is numerically trivial, we have 
     $(K_X+\Delta)\cdot F=0$ implies $ \deg K_F +\Delta\cdot F=0 $ implies $ \Delta\cdot F=2$.
     Hence $\deg [\rddelta^h :V] \leq 2$. Since $u$ is an extremal contraction, we have $\rddelta^v =u^*M$ for some $\mathbb{Q}$-Cartier divisor $M$ on $V$ by \cite[Proposition 3.1]{CHP16}. By Lemma \ref{lem:Canonical bundle} and Lemma \ref{lem:divisor}, we get $K_X+\Delta \sim_{\mathbb{Q}}u^*(K_V+P)$ such that $(V,P)$ is lc. Therefore, $K_{D_i}+\Diff(\Delta-D_1)\sim_{\mathbb{Q}}(u|_{D_i})^*(K_V+P)$ for $i=1,2$. In case $(1)$, since $\deg[D_1:V]=2$, there is a $B$-bimeromorphic involution over $V$. In case of $(3)$, $\deg[D_i:V]=1$, so $D_i$ and $V$ are $B$-bimeromorphic. For more detailed arguments, one can check \cite[Proof of Lemma 4.5, Page No. 20]{Fu25}.

 \end{proof}      
   \end{lemma}
   Now we will prove Proposition \ref{Pro}.
    \begin{proof}[Proof of Proposition \ref{Pro}]
      If $f$ is bimeromorphic, then by \cite[Theorem 17.4]{Kol92} we are in case $(1.1)$ and $(2.1)$. Therefore, we may assume that $\dim Y< \dim X$.\\
 We run $K_X+\Delta -\epsilon \lfloor \Delta \rfloor$-MMP over $Y$ for $0<\epsilon \ll 1$ by Lemma \ref{MMP}. Let $Z$ be the output of this MMP and let $h:Z \to Y$ be the induced morphism. Denote by  $p:X \dashrightarrow Z$ the composition of divisorial contractions and flips in this process.\\
 Assume that $(Z,p_*\Delta - \epsilon p_*\lfloor \Delta \rfloor)$ is a minimal model. Then $K_Z+p_*\Delta - \epsilon p_*\lfloor \Delta \rfloor$is $h$- nef. Since $K_Z+p_*\Delta$ is numerically $h$ trivial, it follows that $-p_*\lfloor \Delta \rfloor$ is $h$-nef. If dim $Y=0$, then $-p_*\lfloor \Delta \rfloor$ is nef. For any  K\"{a}hler form $\omega$, we obtain $(-p_*\lfloor \Delta \rfloor)\omega^{n-1}\geq 0$. On the other hand, since $p_*\lfloor \Delta \rfloor$ is effective and $\omega$ is a K\"{a}hler form we have $(p_*\lfloor \Delta \rfloor)\omega^{n-1}>0$ which yields a contradiction. Hence $p_*\lfloor \Delta \rfloor=0$ implies that $\lfloor \Delta \rfloor=0 $ by Lemma \ref{Conn}.\\
 
Therefore, we may assume that $ 0<\dim Y< \dim X$. Suppose that $p_*(\lfloor \Delta \rfloor)$ has no horizontal component with respect to $h$. Then by \cite[Lemma 4.4]{Fu00}, we obtain the following exact sequence:
 \[ 0 \to h_*\mathcal{O}_Z(-p_*\lfloor \Delta \rfloor)\to \mathcal{O}_Y \to h_*\mathcal{O}_{p_*\lfloor \Delta \rfloor }\to 0\]
  Note that \cite[Lemma 4.4]{Fu00} uses \cite[Theorem 1-2-7]{KMM87}. In our setting we may instead apply \cite[10.15.1 Page No. 121]{Kol95}.\\
 Define  $T:=h(p_*\lfloor \Delta \rfloor)$. Then $\mathcal{O}_{T}=h_*\mathcal{O}_{p_*\lfloor \Delta \rfloor }$. In particular, $p_*(\lfloor \Delta \rfloor) \cap h^{-1}(y)$ is connected for every $y$. We may assume that $(p_*\rddelta)^h\neq 0$. Since every horizontal component intersects every fiber, it intersects in particular any general fiber $F$ of $h$. Note that $F$ is either a K\"ahler curve or a K\"ahler surface. In both cases $-(p_*\lfloor \Delta \rfloor)|_{F} $ is nef on $F$ which leads to a contradiction. Therefore, $p_*(\lfloor \Delta \rfloor) \cap h^{-1}(y)$ is connected for every $y$. By Lemma \ref{Conn}, we are in case $(2.1)$.\\
 
Now suppose that there exists an extremal Fano contraction $u:Z \to V$ over $Y$. Then $-(K_Z+ p_*\Delta - \epsilon p_*\lfloor \Delta \rfloor)$ is $u$- ample.\\
If $\dim V \leq 1$, then $V$ is projective. Since $u$ is a projective morphism, it follows that $Z$ is also projective which implies that $X$ is projective. In these cases, the assertion follows from  \cite[Proposition 2.1]{Fu00}. If $\dim X=2$, then $\dim Z=2$ and hence $\dim V\leq 1$. This implies that $X$ is projective. 
Thus we may assume that $\dim X=3, \dim V=2$. If $\dim Y\geq 1$, then the case $(2.2)$ follows from Lemma \ref{N} and Lemma \ref{Fano}.
Assume that $\dim Y=0$. If $p_*\lfloor \Delta \rfloor$ is connected, then $\lfloor \Delta \rfloor$ is connected by Lemma \ref{Conn} and we are done. Hence we may assume that $p_*\lfloor \Delta \rfloor$ is not connected.
Since every horizontal component intersects every fiber, it must also intersect every vertical component. In particular, if the horizontal part has only one irreducible component, then $p_*\lfloor \Delta \rfloor$  is connected. Hence by Lemma \ref{Fano}, we are in (3). In this case, $\rddelta$  has two connected components, each of these connected component contains a horizontal component. Since $u$ is an extremal Fano contraction, each horizontal component is relatively ample over $V$.\\

\textit{Claim:} $u|_{D_i}:D_i \to V$ is finite for $i=1,2$.\\
By Stein factorization, we obtain the following commutative diagram:

\[\begin{tikzcd}
	& {\bar{V}} \\
	{D_i} && V
	\arrow[from=1-2, to=2-3]
	\arrow["{\bar{u}_{D_i}}", from=2-1, to=1-2]
	\arrow["{u|_{D_i}}"', from=2-1, to=2-3]
\end{tikzcd}\]
It suffices to show $\bar{u}_{D_i}$ is finite. Suppose, for instance $i=1$. If the morphism is not finite, then there exists a fiber of positive dimension. Since the morphism is projective, we may choose a curve $C$ contained in such a fiber. Then $C \subset $ supp $D_1$. Moreover, we have  $D_2.C >0$ and hence  supp $D_1\cap$ supp $D_2\neq \emptyset$. This implies that $p_*\lfloor \Delta \rfloor$ is connected, contradicting our assumption.
Therefore, $u|_{D_i}$ is finite. Since $V$ is normal, $u|_{D_i}$ is an isomorphism by Zariski's main theorem. Consequently, we are in case $(1.2)$.
  \end{proof}

The following lemma is used in the proof of Proposition \ref{Pro}.
\begin{lemma}\label{Conn}
    Let $f:X \to Y,h:Z\to Y,p:X\dashrightarrow Z$ be as in the proof of Proposition \ref{Pro}. Then for every $y\in Y$, number of connected components of $\lfloor \Delta \rfloor \cap f^{-1}(y)$ is equal to the number of connected components of $p_*\lfloor \Delta \rfloor \cap h^{-1}(y)$.
    \begin{proof}
Since $p$ is a composition of flips and divisorial contractions,  it is enough to verify that the number of connected components remains unchanged at each step of the MMP. Hence we may assume that  $p$  is either a flipping contraction  or a divisorial contraction.\\
Since the MMP is $K_X+\Delta$-relatively trivial, by Lemma \ref{N}, the map of every MMP step must be $B$-bimeromorphic. Let $p:(X,\Delta) \dashrightarrow (X',\Delta')$ be the $B$-bimeromorphic map. Choose a common crepant resolution satisfying $K_W+\Gamma \sim_{\mathbb{Q}}g_1^*(K_X+\Delta)\sim g_2^*(K_{X'}+\Delta')$. By Kollar-Shokurov connectedness principle as in \cite[Theorem 17.4]{Kol92} and Lemma \ref{lem:Top}, we obtain our desired result.
    \end{proof}
\end{lemma}

We define admissible and preadmissible sections similarly as defined in \cite{Fu00}.\\
\begin{definition}\cite[Definition 1.5]{Fu00}\label{def:pre&ad}
    Let $(X,\Delta)$ be a compact K\"{a}hler sdlt n-fold and let $m$ be a sufficiently large and divisible integer. Let $\mu: (X',\Theta):\sqcup_{i=1}^{k} (X'_i,\Theta_i)\to X=\cup X_i $ be the normalization morphism. We define inductively on dimension as follows:
\begin{itemize}
    \item $s \in H^0(X,m(K_X+\Delta))$ is called preadmissible if $\mu^*s|_{(\sqcup \lfloor \Theta_i \rfloor)} \in H^0((\sqcup \lfloor \Theta_i \rfloor),m(K_{X'}+\Theta))|_{\sqcup \lfloor \Theta_i \rfloor}$ is admissible.
    \item $s \in H^0(X,m(K_X+\Delta))$ is called admissible if $s$ is preadmissible and $g^*(s|_{X'_i})=s|_{X'_j}$ for every $B$-bimeromorphic map $g:(X'_i,\Theta_i) \dashrightarrow (X'_j,\Theta_j),$ for all $i,j$.
\end{itemize}
\end{definition} 
\begin{lemma}\cite[Lemma 4.2]{Fu00}\label{end}
Let $(X,\Delta)$ be a compact K\"{a}hler slc $n$-fold. Let $\mu:X' \to X$ be the normalization morphism such that $K_{X'}+\Theta\sim_{\mathbb{Q}}\mu^*(K_X+\Delta)$. Let $h:(Y,\Gamma) \to (X',\Theta)$ be a proper bimeromorphic morphism such that $(Y,\Gamma)$ is dlt and  $K_Y+\Gamma \sim _{\mathbb{Q}}h^*(K_{X'}+\Theta)$. Then every section in $\PA(Y,m(K_Y+\Gamma))$ descends to a section on $(X,\Delta)$.
\begin{proof}
    By the definition of slc, the variety $X$ satisfies Serre's condition $S_2$ and has normal crossing in codimension $1$. Hence, the statement follows directly from the definition of preadmissible sections.  For a detailed argument, we refer the reader to  \cite[Lemma 4.9]{Xu19}.

\end{proof}
\end{lemma}

The following theorem is proved by Fujino for projective dlt pairs \cite[Proposition 4.5]{Fu00} . We will prove this result for compact K\"{a}hler dlt pairs.

\begin{theorem}\cite[Proposition 4.5]{Fu00}\label{imp}
   Let $(X,\Delta) := \sqcup_{i=1}^{k} (X_i,\Delta_i)$ be a compact K\"{a}hler $\mathbb{Q}$-factorial dlt pair of dimension at most $3$. Let $m$ be a sufficiently large and divisible positive integer. Assume that
    \begin{enumerate}
        \item $K_X+\Delta$ is nef,
        \item $\A(\lfloor \Delta \rfloor, m(K_X+\Delta)|_{\lfloor \Delta \rfloor})$ generates $\mathcal{O}_{\lfloor \Delta \rfloor}(m(K_{\lfloor \Delta \rfloor }+\Diff (\Delta -\lfloor \Delta \rfloor)))$.
    \end{enumerate}
Then 
\[ \PA(X,m(K_X+\Delta)) \to \A ( \lfloor \Delta \rfloor, m(K_X+\Delta)|_{\lfloor \Delta \rfloor}) \]
 is surjective and $\PA(X,m(K_X+\Delta))$ generates $\mathcal{O}_{X}(m(K_X+\Delta))$.
 \end{theorem} 
 
 \begin{proof}
    We first assume that $X$ is $\mathbb{Q}$-factorial compact K\"ahler irreducible dlt pair.\\
By Theorem \ref{abun}, for a sufficiently large and divisible integer $m$, we have a morphism $f:X \to Y$. If $\lfloor \Delta \rfloor=0$, then the result is trivial because in this case $\PA(X,m(K_X+\Delta))=H^0(X,m(K_X+\Delta))$. Hence we may assume that $\lfloor \Delta \rfloor\neq 0$. \\
We have the following possibilities:
 \begin{enumerate}
 \item[(I)] $f^{-1}(y) \cap \lfloor \Delta \rfloor$ is connected for every $y\in Y$.
     \subitem (a) $\kappa(X,K_X+\Delta)=0$ and $\lfloor \Delta \rfloor$ is connected.
     \subitem (b) $\kappa(X,K_X+\Delta)\geq 1 $ and $f^{-1}(y) \cap  \lfloor \Delta \rfloor$ is connected for every $y \in Y$ and $f(\lfloor \Delta \rfloor)=Y$.
     \subitem (c) $\kappa(X,K_X+\Delta)\geq 1 $ and $f^{-1}(y) \cap  \lfloor \Delta \rfloor$ is connected for every $y \in Y$ and $f(\lfloor \Delta \rfloor)\subsetneq Y$.
     \item[(II)]$f^{-1}(y) \cap \lfloor \Delta \rfloor $ is not connected for some $y \in Y$.\\     
 \end{enumerate}
\textbf{Case (I(a))}: We have the following long exact sequence of cohomology\\
 \begin{equation*}    
 0 \to H^0(X,\mathcal{O}_X(m(K_X+\Delta)-  \lfloor \Delta \rfloor)) \to H^0(X,\mathcal{O}_X(m(K_X+\Delta))) \to H^0(\lfloor \Delta  \rfloor, m(K_X+\Delta)|_{\lfloor \Delta \rfloor}) 
 \end{equation*}
 
 Since $m$ is sufficiently large and divisible and $\mathcal{O}_X(m(K_X+\Delta))$ is globally generated, we can choose $s\in H^0(X,\mathcal{O}_X(m(K_X+\Delta)))$ such that $s$ does not vanish on $\lfloor \Delta \rfloor$, i.e., $0\neq s|_{\lfloor \Delta \rfloor}\in H^0(\lfloor \Delta  \rfloor, m(K_X+\Delta)|_{\lfloor \Delta \rfloor})$. Since, $\kappa(K_X+\Delta)=0, \mathcal{O}_X(m(K_X+\Delta))\cong \mathcal{O}_X $ and $\mathcal{O}_{\lfloor \Delta \rfloor}(m(K_{\lfloor \Delta \rfloor} +\Diff (\Delta -\lfloor \Delta \rfloor)))\cong \mathcal{O}_{\lfloor \Delta \rfloor}.$ Since $\lfloor \Delta \rfloor$ is connected, third term is also $1$-dimensional. Thus we get the required result.\\

\textbf{Case (I(b))}:
    From hypothesis $(2)$, we can construct a morphism $\phi: \lfloor \Delta \rfloor \to Z $ with connected fibers by the linear system corresponding to $\A(\lfloor \Delta \rfloor, m(K_X+\Delta)|_{\lfloor \Delta \rfloor})$. Since every  fiber of $f|_{ \lfloor \Delta \rfloor }$ is also contracted by $\phi$, by Lemma \ref{rig} there exists a morphism $\psi:Y \to Z$ such that $\psi\circ f|_{ \lfloor \Delta \rfloor }= \phi$. Now for $s\in \A(\lfloor \Delta \rfloor, m(K_X+\Delta)|_{\lfloor \Delta \rfloor})$, there exists a section $t$ on $Z$ such that $s=\phi^*t$. Define $u:= f^*\psi^*t \in H^0(X,m(K_X+\Delta))$. Note that $u|_{\lfloor \Delta \rfloor}=s \in \A(\lfloor \Delta \rfloor, m(K_X+\Delta)|_{\lfloor \Delta \rfloor}) $ implies that $ u\in \PA(X,m(K_X+\Delta))$. Hence, $\PA(X,m(K_X+\Delta)) \to \A ( \lfloor \Delta \rfloor, m(K_X+\Delta)|_{\lfloor \Delta \rfloor})$ is surjective. By \cite[Lemma 4.3]{Fu00}, we obtain $\PA(X,m(K_X+\Delta))$ generates $\mathcal{O}_{X}(m(K_X+\Delta))$.\\
    
\textbf{Case (I(c))}: 
By  \cite[Lemma 4.4]{Fu00}, we get the following exact sequence
 $$0 \to f_*\mathcal{O}_X(-\lfloor \Delta \rfloor)\to \mathcal{O}_Y \to f_*\mathcal{O}_{\lfloor \Delta \rfloor }\to 0$$
 Note that \cite[Lemma 4.4]{Fu00} uses \cite[Theorem 1-2-7]{KMM87}. In our setting we can use \cite[10.15.1 Page-121]{Kol95}.\\
 
Let $T$ be the complex analytic subspace of $Y$ associated to the ideal sheaf $\mathscr{I}_{T}:= f_*\mathcal{O}_X(-\lfloor \Delta \rfloor)$. Then
 we have $\mathcal{O}_{T}=f_*\mathcal{O}_{\lfloor \Delta \rfloor }$. By \cite[Proposition 4.5]{ambro1998locus} $T$ is seminormal. Similarly as in case I(b), we construct $\phi: \lfloor \Delta \rfloor \to Z $ and get a morphism $\psi:T \to Z$. Then by \cite[Lemma 4.3]{Fu00},
we can pull back any admissible section on $\lfloor \Delta \rfloor $ to a preadmissible section on $X$. Again by \cite[Lemma 4.3]{Fu00} $\PA(X,m(K_X+\Delta))$ generates $\mathcal{O}_{X}(m(K_X+\Delta))$.\\

\textbf{Case (II)}:
From Proposition \ref{Pro}, $(X,\Delta)$ is a generically $\mathbb{P}^1$- bundle over $(V,P)$. Let $f:(X,\Delta)\to Y$ be the morphism defined by Theorem \ref{abun} and $u:(X',\Delta')\to (V,P)$ be the extremal Fano contraction over $Y$. Since log canonical ring is preserved under the steps of MMP, we have \\
$$ H^0(X,m(K_X+\Delta))\cong H^0(X',m(K_{X'}+\Delta'))   $$
By Lemma \ref{N}  we know that the composition map $p:X \dashrightarrow X'$ is $B$-bimeromorphic. Choose a common log resolution of $(X, \Delta)$ and $(X', \Delta')$ which also resolves the map $p$. Let $\alpha:(W,\Gamma)\to (X,\Delta)$ and $\beta:(W,\Gamma)\to (X',\Delta')$ be the induced morphisms; then we have $K_W+\Gamma\sim_{\mathbb{Q}}\alpha^*(K_X+\Delta)\sim_{\mathbb{Q}}\beta^*(K_{X'}+\Delta')$. Let $\Gamma =\sum d_i\Gamma_i $. We define $\Gamma^b:=\sum_{d_i=1}\Gamma_i$.
Since both $(X,\Delta)$ and $(X,\Delta')$ are $\mathbb{Q}$-factorial dlt pair, both $\lfloor \Delta \rfloor$ and $\lfloor \Delta' \rfloor$ are seminormal \cite[Remark 1.2- (2),(3)]{Fu00}. Since $(W,\Gamma)$ is log smooth Nklt$(W,\Gamma)= \Gamma^b $, by connectedness lemma \cite[Theorem 17.4]{Kol92} we get surjective morphisms $\Gamma^b\to \lfloor \Delta \rfloor $ and $\Gamma^b \to \lfloor \Delta' \rfloor $ with connected fibers. Note that $ \alpha^{-1}(\rddelta)=\beta^{-1}(\lfloor \Delta' \rfloor)=\Gamma^b$. So we have $\alpha_*\mathcal{O}_{\Gamma^b}=\mathcal{O}_{\lfloor \Delta \rfloor}$ and $\beta_*\mathcal{O}_{\Gamma^b}=\mathcal{O}_{\lfloor \Delta' \rfloor}$. Then we have 
\begin{equation}\label{eqn}
    H^0(\lfloor \Delta \rfloor, m(K_X+\Delta)|_{\lfloor \Delta \rfloor})\cong H^0(\Gamma^b,m(K_W+\Gamma)|_{\Gamma^b}) \cong H^0(\lfloor \Delta' \rfloor, m(K_{X'}+\Delta')|_{\lfloor \Delta' \rfloor}).
    \end{equation}
 Let $s\in \A(\lfloor \Delta  \rfloor, m(K_X+\Delta)|_{\lfloor \Delta \rfloor})$. By the isomorphism $(4.1)$, $s$ is also an element of $H^0(\lfloor \Delta' \rfloor, m(K_{X'}+\Delta')|_{\lfloor \Delta' \rfloor})$. Let $\lfloor \Delta' \rfloor= \lfloor\Delta'\rfloor^h +\lfloor\Delta'\rfloor^v$ where $\lfloor\Delta'\rfloor^h$ and $\lfloor\Delta'\rfloor^v$ be the horizontal and vertical component respectively with respect to the morphism $u$. 
  Since $s$ is admissible, $s|_{\lfloor\Delta'\rfloor^h}$ is $B$- bimeromorphic involution invariant over $(V,P)$ and hence it descends to a section $t$ on $(V,P)$. By the isomorphism 
  \[ H^0(X,m(K_X+\Delta))\cong H^0(X',m(K_{X'}+\Delta'))\cong H^0(V,m(K_V+P)), \] we can pull back the section $t$ to a section $w\in  H^0(X',m(K_{X'}+\Delta')) $.\\
\textit{Claim:} $s=w|_{ \lfloor \Delta' \rfloor}$.\\
By Lemma \ref{biholo}, we can choose a small analytic open set $U$  in $V$ such that $u^{-1}(U)$ is biholomorphic to $U\times \mathbb{P}^1$ and $u|_{u^{-1}(U)}:U\times \mathbb{P}^1\to U$ is the first projection. We may replace $V$ by the open set $U$ and $X$ by $U\times \mathbb{P}^1$  and the section $t$ corresponds to a curve $C$ on $V$. Now we can view two sections $s$ and $w$ as section for the restriction map $C\times \mathbb{P}^1 \to C$ as both the sections $s$ and $w$ are pullback of $t$ under the morphism $u$. Then we can directly apply the same argument as in \cite[12.3.4]{Kol92}. Then assuming m is even we get $s|_{{\lfloor \Delta' \rfloor}^h}=w|_{{\lfloor\Delta'\rfloor}^h}$.\\

Now we check $s|_{\lfloor\Delta'\rfloor^v}=w|_{\lfloor\Delta'\rfloor^v}$. Let $\lfloor\Delta'\rfloor^v= \sum E_i$. Since $u$ is a contraction of negative extremal ray, by contraction theorem we have $E_i=u^*D_i$ for every $i$. It is enough to check that $s|_{E_i}=w|_{E_i}$ for every $i$. Since each $E_i$ is pullback of a divisor from $V$, if a fiber intersects the vertical part, then it is contained in the  support of $E_i$. On the other hand, horizontal part intersects with every fiber. Hence we have $\lfloor\Delta'\rfloor^h \cap E_i \neq \emptyset$. We can choose  an irreducible component, $E_i'$ of $E_i \cap \lfloor\Delta'\rfloor^h$ such that $u|_{E_i'}:E_i' \to D_i$ is dominant. Consider the following commutative diagram:

=
\[\begin{tikzcd}
	{H^0(E_i,m(K_{X'}+\Delta')|_{E_i})} \\
	&& {H^0(E_i',m(K_{X'}+\Delta')|_{E_i'})} \\
	{H^0(D_i,m(K_V+P)|_{D_i})}
	\arrow[from=1-1, to=2-3]
	\arrow["\cong", from=3-1, to=1-1]
	\arrow["(u|_{E_i'})^*",hook, from=3-1, to=2-3]
\end{tikzcd}\]

Since $u^*D_i=E_i,u|_{E_i}$ has connected fibers. Since $D_i\subset \lfloor P \rfloor, D_i$is seminormal \cite[Proposition 3.2]{ambro1998locus} and  hence the vertical map is an isomorphism by Lemma \ref{semi}. The map $(u|_{E_i'})^*$ is injective since $u|_{E_i'}:E_i'\to D_i$ is dominant. Since $E_i \subset \lfloor\Delta'\rfloor'^h$,for every $E_i'$ we get $s|_{E_i'}=w|_{E_i'}$, so we get $s|_{E_i}=w|_{E_i}$. Thus we are done.\\
 From given hypothesis $ \A(\lfloor \Delta \rfloor, m(K_X+\Delta)|_{\lfloor \Delta \rfloor})$ generates $\mathcal{O}_{\lfloor \Delta \rfloor}(m(K_{\lfloor \Delta \rfloor } $+Diff $(\Delta -\lfloor \Delta \rfloor)$, the restriction to $\lfloor \Delta'\rfloor^h$ generates $\mathcal{O}_{\lfloor \Delta' \rfloor^h}(m(K_{X'}+\Delta')|_{\lfloor \Delta' \rfloor^h})$. Therefore 
$\PA(X,m(K_X+\Delta))$ generates $\mathcal{O}_{X}(m(K_X+\Delta))$.

Now suppose $(X,\Delta)=\sqcup_{i=1}^k (X_i,\Delta_i)$. By definition, $\rddelta=\sqcup_{i=1}^k\lfloor\Delta_i\rfloor$
Let $t\in \A(\rddelta, (K_X+\Delta)|_{\rddelta})$. Then $t=(t_1,...,t_k)$ where $t_i$ is admissible on $\lfloor\Delta_i\rfloor$. Since surjectivity is already proved for each irreducible component, we obtain $s=(s_1,...,s_k)$ such that $s|_{\rddelta}=t$. Then by definition, $s$ is preadmissible. Thus surjectivity follows. Then by \cite[Lemma 4.3]{Fu00} we can show $\PA(X, m(K_X+\Delta))$ generates the line bundle.
 \end{proof}

\begin{lemma} \label{lem:admissible} 
Let $(X,\Delta)$ be a compact K\"{a}hler pure dimensional dlt pair of dimension  2 such that the pluricanonical representation is finite. Assume that $K_X+\Delta$ is nef. Then $\A(X,m(K_X+\Delta))$ generates $\mathcal{O}_X(m(K_X+\Delta))$ for sufficiently large and divisible integer $m$.
\end{lemma}  
\begin{proof}
Let $s\in \A(\lfloor \Delta  \rfloor, m(K_X+\Delta)|_{\lfloor \Delta \rfloor})$. By Proposition \ref{imp}, there exists a section  $s'\in \PA(X,m(K_X+\Delta))$ such that $s'|_{\lfloor \Delta  \rfloor}=s $. Define  $t:= \frac{1}{|G|}\sum_{g \in G} g^*s'$ where $G=\rho_m(\Bim(X,\Delta))$. Then $g^*t =t$ for every $g\in G$, and hence $t$ is admissible. We have, $t|_{\lfloor \Delta \rfloor}=s$ by Lemma \ref{restriction}.
 Therefore, the natural restriction morphism
  $\A(X,m(K_X+\Delta)) \to \A(\lfloor \Delta  \rfloor, m(K_X+\Delta)|_{\lfloor \Delta \rfloor})$ is surjective.\\
  
  Let $q:=|\rho_m(\Bim(X,\Delta))|< \infty$. Let $\sigma_i$ be the $i$-th elementary symmetric polynomial. Since identity morphism is also an automorphism and the only common zero  of all elementary symmetric polynomials in $q$ variables is trivial, we have 
   \[ \cap_{i=1}^{q} \{\sigma_i(g_1^*s,.....,g_q^*s)=0\}=\cap_{j=1}^{q}\{g_j^*s=0\} \subseteq \{s=0\}\]
  Moreover, each $\sigma_i(g_1^*s,.....,g_q^*s)$ is admissible for  $i\in \{ 1,2,...,q\}$. Hence, for $s\in \PA(X,m(K_X+\Delta))$ we may define,
   $$ s(i):= \sigma_i^{\frac{q!}{i}}(g_1^*s,.....,g_q^*s)\in \A(X,q!m(K_X+\Delta)) $$
  Given any point $x\in X$, choose a  section $s\in \PA(X,m(K_X+\Delta)) $ which does not vanish at $x$. Then there exists some $i$ such that $s(i)$ does not vanish at $x$. This shows that $\A(X,q!m(K_X+\Delta))$ generates $\mathcal{O}_{X}(q!m(K_X+\Delta))$, completing the proof.

\end{proof}
\begin{remark}
    For any lc curve $(C,B)$, Lemma \ref{lem:admissible} holds, since the pluricanonical representation is finite.
\end{remark}
\begin{lemma}\cite[Lemma 4.9]{Fu00}\label{restriction}
     Let $(X,\Delta)$ be a compact K\"{a}hler pure dimensional dlt pair of arbitrary dimension such that $K_X+\Delta$ is nef and let $m$ be a sufficiently large and divisible integer. Let $G=\rho_m(\Bim(X,\Delta))$. If $s\in \PA(X,m(K_X+\Delta))$ then $g^*s|_{\lfloor \Delta \rfloor}=s|_{\lfloor \Delta \rfloor}$ and $g^*s \in \PA(X,m(K_X+\Delta))$. Let $G=\rho_m(\Bim(X,\Delta))$. If $G$ is finite, then 
     \[ \sum_{g\in G} g^*s \in \A(X,m(K_X+\Delta))\]
     \[ \prod_{g\in G}g^*s \in \A(X,m|G|(K_X+\Delta))\]
     and 
     \[ \prod_{g\in G} g^*s|_{\rddelta}= (s|_{\rddelta})^{|G|}\].
     \end{lemma}
    \begin{proof}
We will  follow the argument of \cite[Lemma 4.9]{Fu00} directly. Since the argument there uses only the connectedness Lemma \cite[Theorem 17.4]{Kol92}  and existence of resolution of singularities - both valid in our setting, so the same proof works in our case as well.  
    \end{proof}

\section{Minimally Preadmissible and Minimally Admissible Sections}\label{sec:sdlt-abundance}
In this section, we introduce the notions of minimally admissible and minimally preadmissible sections (see Definition \ref{def:ma-mpa}), which generalize admissible and preadmissible sections. We then establish several results which will be used in the proof of our main Theorem \ref{thm:mainimp}.

We start with a slc pair $(X,\Delta):=\cup_{i=1}^{k}(X_i,\Delta_i)$ of dimension 3. Let $\mu:(X',\Delta'+D'):=\sqcup_{i=1}^{k}(X_i',\Delta_i'+D'_i)\to (X,\Delta)$ be the normalization morphism, where $D'$ is the conductor divisor on $X'$. Assume that $\nu:D^n \to D'$ is the normalization. Then there is a Galois involution $\tau:D^n \to D^n$ over $D'$. Let $\phi:(Z,\Gamma):=\sqcup_{i=1}^{k} (Z_i,\Gamma_i) \to (X',\Delta'+D')$ be a $\mathbb{Q}$ factorial dlt modification of $(X', \Delta'+D')$. Then by adjunction $K_{\rdgamma}+\Diff(\Gamma-\rdgamma))\sim_{\mathbb{Q}}(K_Z+\Gamma)|_{\rdgamma}$ such that $(\rdgamma,\Diff(\Gamma-\rdgamma))$ has sdlt singularities. \\

Write $\rdgamma =\cup_{i=1}^\ell T_i$; then by adjunction we may assume that  $(K_{\rdgamma}+\Diff(\Gamma-\rdgamma))|_{T_i} \sim_{\mathbb{Q}}K_{T_i}+\Theta_{T_i}$ for all $i\in \{1,2,\ldots, \ell\}$.
Let $\nu_{\rdgamma}:(\widetilde{\rdgamma},\Theta):= \sqcup_{i=1}^\ell(T_i,\bar\Theta_{T_i})=\sqcup_{i=1}^\ell(T_i,\Theta'_{T_i}+\bar{D}_{T_i})\to \cup_{i=1}^\ell (T_i,\Theta_{T_i})$ be the normalization morphism; we note that since each $T_i$ is normal (as $(\rdgamma, \Diff(\Gamma-\rdgamma))$ is sdlt), the normalization $\nu_{\rdgamma}$ only affects the boundary divisor.\\
In order to prove that $\PA(Z,m(K_Z+\Gamma))$ generates $\mathcal{O}_{Z}(m(K_Z+\Gamma))$, by Theorem \ref{imp}, it is enough to show that $\A(\rdgamma,\Diff(\Gamma-\rdgamma))$ generates $\mathcal{O}_{\rdgamma}(m(K_{\rdgamma}+\Diff(\Gamma-\rdgamma)))$. Since  the finiteness of pluricanonical representation fails on K\"ahler varieties (see Remark \ref{K3}), there are cases in which $\A(\rdgamma,\Diff(\Gamma-\rdgamma)$ fails to generate $\mathcal{O}_{Z}(m(K_Z+\Gamma))$. In particular,  $\PA(Z,m(K_Z+\Gamma))$  does not generate $\mathcal{O}_{Z}(m(K_Z+\Gamma))$.
To remedy this situation we introduce a new collection of sections which contains the preadmissible sections.

 \begin{definition}\label{def:bad}
Let $\phi:(Z,\Gamma) \to (X',\Delta'+D')$ be defined as in the above paragraph. Let $T \subset \rdgamma$ be an irreducible component.
We say 
\begin{enumerate}
    \item $T$ is \textit{bad}, if the pluricanonical representation of the pair$(T,\bar{\Theta}_T)$ is infinite.
    \item $T$ is \textit{good}, if it is not bad, i.e.,  if the pluricanonical representation of the pair $(T,\bar{\Theta}_T)$ is finite.
\end{enumerate}
 If $\rdgamma$ contains a bad component, then we say that the pair $(Z, \Gamma)$ is a \textit{bad} pair.
\end{definition}
\begin{remark}\label{rmk:bad}
Note that if $S$ and $T$ are two components of $\rdgamma$ and $\eta:(S,\bar\Theta_{S}) \dashrightarrow (T,\bar\Theta_T)$ is a $B$-bimeromorphic map, then $\Bim(S,\bar\Theta_S)\cong \Bim(T,\bar\Theta_T)$ by Remark \ref{rmk:b-bimeromorphic}. Therefore the pluricanonical representation is finite for the pair $(S,\bar\Theta_{S})$ if and only if it is finite for the pair $(T,\bar\Theta_T)$. This implies that $S$ is a bad component if and only if $T$ is a bad component.
\end{remark}
\begin{lemma}\label{lem:bad}
    With the same notation as in Definition \ref{def:bad}, if $T$ is a \textit{bad} component of $\rdgamma$, then the following hold:
    \begin{enumerate}
        \item $T\cap \Supp(\rdgamma-T)=\emptyset$, and 
        \item the Kodaira dimension $\kappa(T)=0$ and $\bar{\Theta}_T=0$.
    \end{enumerate}
\end{lemma}
\begin{proof}
(1) Assume that $T\cap \Supp(\rdgamma-T)\neq\emptyset$. Then $\bar{\Theta}_T\neq 0$, where $(K_T+\bar\Theta_T)\sim_{\mathbb Q} (K_{\rdgamma}+\Diff(\Gamma-\rdgamma))|_{T}$ is defined by adjunction (see above). Then by Theorem \ref{plu}, the pluricanonical representation of the pair $(T,\bar{\Theta}_T)$ is finite. This is a contradiction.\\

(2)  Since $T$ is a bad component, by Theorem \ref{plu}, $\bar\Theta_T=0$. Moreover, in this case from Theorem \ref{plu}, it follows that the Kodaira dimension $\kappa(T)=0$.

\end{proof}
Recall that $\nu_{\rdgamma}:\widetilde{\rdgamma} \to \rdgamma$ is the normalization morphism and $\widetilde{
\lfloor \Gamma \rfloor}:= \sqcup\; T_i$.
Now we define generalizations of admissible and preadmissible sections.
\begin{definition}\label{def:ma-mpa}
With the same notation and hypotheses as in Definition \ref{def:bad}, we fix a sufficiently large and divisible positive integer $m$. The \textit{minimally  admissible} sections ${\rm MA}\bigl(\rdgamma, m(K_{\rdgamma}+\Diff_{\rdgamma}(\Gamma-\rdgamma)\bigr)\bigr)$ is defined as the collection of sections  $s\in H^0\bigl(\rdgamma, m(K_{\rdgamma}+\Diff_{\rdgamma}(\Gamma-\rdgamma)))$
 satisfying the following properties:

\begin{enumerate}
\item $s$ is preadmissible (as in Definition \ref{def:pre&ad}).
\item If $\theta \colon (T_i,\bar\Theta_{T_i})\dashrightarrow (T_j,\bar\Theta_{T_j})$ is a $B$-bimeromorphic map of two good components of $\rdgamma$ (not necessarily distinct), then $\theta^*(s|_{T_j}) = s|_{T_i}$ holds.

\item If $T_i$ is a bad component of $\rdgamma$ and $\theta:(T_i,\bar\Theta_{T_i})\dashrightarrow (T_i,\bar\Theta_{T_i})$ is a $B$-bimeromorphic map such that $\theta^2=\mbox{id}$, then $\theta^*(s|_{T_i}) = s|_{T_i}$ holds.

\item If the Galois involution $\tau: D^n\to D^n$ (defined above) induces a $B$-bimeromorphic map on two distinct bad components $T_i$ and $T_j$, $i\neq j$, of $\rdgamma$, i.e., $\tau|_{T_i}:(T_i,\bar\Theta_{T_i})\dashrightarrow (T_j,\bar\Theta_{T_j})$ is $B$-bimeromorphic, then $\tau^*(s|_{T_j}) = (s|_{T_i})$. Note that in this case $T_i$ and $T_j$ are bimeromorphic to two components $D^n_1$ and $D^n_2$ of $D^n$ induced by the dlt model $\phi$ defined above.

\end{enumerate}
We further define \textit{minimally preadmissible} sections as follows:
\[
\operatorname{MPA}\Bigl(
Z,m(K_Z+\Gamma\bigr)
\Bigr)
:=
\left\{
t \in H^0(Z,m(K_Z+\Gamma)) \;:\; t|_{\rdgamma}\in \operatorname{MA}(\rdgamma,m(K_{\rdgamma}+\Diff_{\rdgamma}(\Gamma-\rdgamma)))
\right\}
\]

\end{definition}

Now we prove a new result (see Proposition \ref{pro:bad-abun} below) which is an analog of Theorem \ref{imp}, where $\operatorname{MPA}(Z,m(K_Z+\Gamma))$ plays the role of $\PA(Z,m(K_Z+\Gamma))$, and $\operatorname{MA}(\lfloor \Gamma \rfloor,
\, m\bigl(K_{\lfloor \Gamma \rfloor}
+ \operatorname{Diff}_{\lfloor \Gamma\rfloor}(\Gamma - \lfloor \Gamma \rfloor)))$ is that of  $\A(\lfloor \Gamma \rfloor,
\, m\bigl(K_{\lfloor \Gamma \rfloor}
+ \operatorname{Diff}_{\lfloor \Gamma \rfloor}
(\Gamma - \lfloor \Gamma \rfloor)))$. 
By definition, we have the natural inclusions
\[
\begin{aligned}
\A\Bigl(
\lfloor \Gamma \rfloor,
m\bigl(K_{\lfloor \Gamma \rfloor}
+\operatorname{Diff}_{\lfloor \Gamma \rfloor}
(\Gamma-\lfloor \Gamma \rfloor)\bigr)
\Bigr)
&\subset
\operatorname{MA}\Bigl(
\lfloor \Gamma \rfloor,
m\bigl(K_{\lfloor \Gamma \rfloor}
+\operatorname{Diff}_{\lfloor \Gamma \rfloor}
(\Gamma-\lfloor \Gamma \rfloor)\bigr)
\Bigr) \\
&\subset
\PA\Bigl(
\lfloor \Gamma \rfloor,
m\bigl(K_{\lfloor \Gamma \rfloor}
+\operatorname{Diff}_{\lfloor \Gamma \rfloor}
(\Gamma-\lfloor \Gamma \rfloor)\bigr)
\Bigr)
\end{aligned}
\]
\begin{center}
    and
\end{center}
\[
\PA\bigl(Z,m(K_Z+\Gamma)\bigr)
\subset
\operatorname{MPA}\bigl(Z,m(K_Z+\Gamma)\bigr).
\]
\begin{remark}\label{rmk:bad2}
Note that
if $\rdgamma$ contains no bad component, then a section $s$ is minimally admissible if and only if $s$ is admissible, and it is minimally preadmissible if and only if it is preadmissible.
\end{remark}
\begin{proposition}\label{pro:bad-abun}
     With the same notation and hypotheses as in Definition \ref{def:bad}, let $(Z,\Gamma):=\sqcup_{i=1}^{k} (Z_i,\Gamma_i)$ be a  pair. If  $K_Z+\Gamma$ is nef, then for $m$ sufficiently large and divisible the following hold:
     \begin{enumerate}
         \item [(A)]$\MA(\lfloor \Gamma \rfloor, 2m(K_Z+\Gamma)|_{\lfloor \Gamma \rfloor})$ generates $\mathcal{O}_{\lfloor \Gamma \rfloor}(2m(K_{\lfloor \Gamma \rfloor }+\Diff (\Gamma -\lfloor \Gamma\rfloor)))$, and 

         \item [(B)] Assume that the Kodaira dimension $\kappa(Z_i, K_{Z_i}+\Gamma_i)\geq 1$ for all $i\in\{1,2...,k\}$. Then  the restriction map $\MPA(Z,2m(K_Z+\Gamma)) \to \MA(\lfloor \Gamma \rfloor, 2m(K_Z+\Gamma)|_{\lfloor \Gamma \rfloor}) $ is surjective, and hence
     $\MPA(Z,2m(K_Z+\Gamma))$ generates $\mathcal{O}_{Z}(2m(K_Z+\Gamma))$.
     \end{enumerate}

\end{proposition}
 \begin{proof}

 (A) Recall that $\nu_{\rdgamma}: \widetilde{\rdgamma}\to \rdgamma$ is the normalization morphism. Observe that a line bundle $\mathcal L$ on $\widetilde{\rdgamma}$ is globally generated if and only if $\mathcal L$ is globally generated on each irreducible component  of $\widetilde{\lfloor \Gamma \rfloor}$ (as the components are pairwise disjoint). Assume that $\mathcal L:=\nu_{\rdgamma}^*( \mathcal{O}_{\lfloor \Gamma \rfloor}(m(K_{\lfloor \Gamma \rfloor }+\Diff (\Gamma -\lfloor \Gamma\rfloor))))$. 
   Our goal is to identify a special collection of sections of $H^0(\widetilde{\rdgamma},\mathcal{L})$ which generate $\mathcal L$ and descend to minimally admissible sections on $\rdgamma$, i.e., for every point $x\in \Supp(\widetilde{\rdgamma})$ find a section $s\in H^0(\widetilde{\rdgamma},\mathcal{L}) $ such that $s(x)\neq 0$ and $s$ descends to a minimally admissible section on $\rdgamma$. To identify this collection we will work with global sections of $\mathcal L|_T$ which generate $\mathcal L|_T$, where $T$ is an irreducible component of $\widetilde{\rdgamma}$.
 First observe that for every good component $T$ of $\widetilde{\rdgamma}$, the admissible sections of $\mathcal{L}|_{T}$ generate the line bundle $\mathcal{L}|_{T}$ on $T$ by Lemma \ref{lem:admissible}. Thus we only need to focus on the bad components. Let $T$ be any such bad component. Then by Lemma \ref{lem:bad}, we know that the Kodaira dimension $\kappa(T)=0$, and hence, $H^0(T, mK_T)=\mathbb{C}$ for all sufficiently large and divisible $m\in \mathbb{N}$. Let $0\neq s\in H^0(T, mK_T)$; then $s^2\in H^0(T,2mK_T)$ generates the line bundle $\mathcal{L}^2|_{T}$. We claim that $s^2$
is invariant under every self $B$-bimeromorphic involution $\sigma:T \dashrightarrow T$. Indeed, let $\sigma$ be a $B$-bimeromorphic involution on $T$. Observe that $\sigma^*s=\lambda s$ for some $\lambda \in \mathbb{C}^{\times}$, as $H^0(T, mK_T)$ is $1$-dimensional. Hence, $\lambda^2=1$, as $\sigma^2=\mbox{id}$. Therefore $\sigma^*s^2=s^2$; this proves our claim.

Now we will show that $\MA(\rdgamma, 2m(K_Z+\Gamma)|_{\rdgamma})$ generates the line bundle $\mathcal{O}_{\lfloor \Gamma \rfloor}(2m(K_{\lfloor \Gamma \rfloor }+\Diff (\Gamma -\lfloor \Gamma\rfloor)))$.
Let $P$ and $Q$ be the supports of all bad components and all good components, respectively. By Lemma \ref{lem: bad0}, we can write $\rdgamma=P \sqcup Q$.
Note that $\widetilde{\rdgamma}= \widetilde{P}\sqcup \widetilde{Q}$.
We have following two cases:\\

\textbf{Case I:} Let $\widetilde{Q}=\sqcup S_i$ where each $S_i$ is good.
Since finiteness of pluricanonical representation holds for good components, admissible sections on $\widetilde{Q}$ generate the line bundle $\mathcal{L}|_{\widetilde{Q}}$ by Lemma \ref{lem:admissible}. In particular, minimally admissible sections generate the line bundle as $\A(\widetilde{Q}, \mathcal{L}|_{\widetilde{Q}}) \cong \MA( \widetilde{Q}, \mathcal{L}|_{\widetilde{Q}}) $ by Remark \ref{rmk:bad2}.\\

\textbf{Case II:}
Let $\widetilde{P}=\sqcup_{i=1}^r T_i$ where each $T_i$ is bad.
Define a section $w$ of $\mathcal{L^2}|_{\widetilde{P}}$ satisfying the following:
    $w|_{T} = \theta^*(w|_{S})$ if there exists a $B$-bimeromorphic map $\theta:T\dashrightarrow S$ induced by the Galois involution. Otherwise, we just take any non zero section. 
Since by Lemma \ref{lem: bad0}(2), no boundary curve occurs on $T_i$, the section $w$ is preadmissible on $\widetilde{P}$.

\begin{claim}
    \[\MA(\widetilde{\rdgamma}, \mathcal{L})= \MA(\widetilde{P},\mathcal{L}|_{\widetilde{P}}) \bigoplus \MA( \widetilde{Q}, \mathcal{L}|_{\widetilde{Q}})\]
\end{claim}

\begin{proof}
    By definition, it is obvious that $\MA(\widetilde{\rdgamma}, \mathcal{L)} \subset \MA(\widetilde{P},\mathcal{L}|_{\widetilde{P}})\bigoplus \MA( \widetilde{Q}, \mathcal{L}|_{\widetilde{Q}})$.
Let $(v,w) \in \MA(\widetilde{P},\mathcal{L}|_{\widetilde{P}})\bigoplus \MA( \widetilde{Q}, \mathcal{L}|_{\widetilde{Q}})$. Since no boundary curve occurs on bad component $(v,w)$ is preadmissible on $\widetilde{\rdgamma}$ as it is preadmissible on $\widetilde{Q}$. Therefore $(v,w)\in \MA(\widetilde{\rdgamma}, \mathcal{L)}$ by Remark \ref{rmk:bad}.
\end{proof}
This shows that $\MA(\widetilde{\rdgamma}, \mathcal{L)}$ generates the line bundle $\mathcal{L}$.
By Lemma \ref{lem:descend}, these sections descend to minimally admissible sections on $\rdgamma$ and, hence we obtain our desired result (A).

 (B) First we will show that for each $i\in \{1,2...,k\}$, $\MPA(Z_i,2m(K_{Z_i}+\Gamma_i))\to \MA(\lfloor \Gamma_i\rfloor, 2m(K_{Z_i}+\Gamma_i)|_{\lfloor\Gamma_i\rfloor})$ is surjective, and hence $\MPA(Z_i,2m(K_{Z_i}+\Gamma_i))$ generates $\mathcal{O}_{Z_i}(2m(K_{Z_i}+\Gamma_i))$. Then we will show that the restriction map $\mbox{MPA}(Z,2m(K_Z+\Gamma)) \to \mbox{MA}(\lfloor \Gamma \rfloor, 2m(K_Z+\Gamma)|_{\lfloor \Gamma \rfloor}) $ is surjective, and hence $\mbox{MPA}(Z,2m(K_Z+\Gamma))$ generates $\mathcal{O}_{Z}(2m(K_Z+\Gamma))$.
 
To see this, assume that $(Z,\Gamma)$ is a $\mathbb{Q}$-factorial irreducible dlt pair of dimension $3$. 
 Since $K_Z+\Gamma$ is nef, it is semiample by Theorem \ref{abun}. Let $f:Z \to R$ be the semiample fibration. We will do case by case analysis as  in Theorem \ref{imp}. We have the following possibilities:
 \begin{itemize}
 \item (I) $\kappa(Z,K_Z+\Gamma)\geq 1 $ and $f^{-1}(r) \cap  \lfloor \Gamma \rfloor$ is connected for every $r \in R$ and $f(\lfloor \Gamma \rfloor)=R$.
     \item (II) $\kappa(Z,K_Z+\Gamma)\geq 1 $ and $f^{-1}(r) \cap  \lfloor \Gamma \rfloor$ is connected for every $r \in R$ and $f(\lfloor \Gamma\rfloor)\subsetneq R$.
     \item(III)$f^{-1}(r) \cap \lfloor \Gamma \rfloor $ is not connected for some $r \in R$.   
 \end{itemize}

     \textbf{Case (I)}: Consider the restriction morphism $f|_{\rdgamma}: \rdgamma \to R$. Since the fibers of $f|_{\rdgamma}$ are connected and $R$ is also connected, it follows from Lemma \ref{lem:Top} that $\rdgamma$ is also connected. In this case the proof is identical to that of Case I(b) of Theorem \ref{imp} once we   replace $\PA(Z,2m(K_Z+\Gamma))$ and $\A(\lfloor \Gamma \rfloor,
\, 2m\bigl(K_{\lfloor \Gamma \rfloor}
+ \operatorname{Diff}_{\lfloor \Gamma \rfloor}
(\Gamma - \lfloor \Gamma \rfloor)))$ 
by $\operatorname{MPA}(Z,2m(K_Z+\Gamma))$ and $\operatorname{MA}(\lfloor \Gamma \rfloor,
\, 2m\bigl(K_{\lfloor \Gamma \rfloor}
+ \operatorname{Diff}_{\lfloor \Gamma\rfloor}(\Gamma-\rdgamma)))$, respectively.\\
. \\
  \noindent
     \textbf{Case (II)}: In this case the proof is identical to that of Case I(c) of Theorem \ref{imp} for the same reason.\\
     
\noindent
\textbf{Case III}: We follow the same argument as in the Case II of Theorem \ref{imp}. 
Recall that in this case we run a $K_Z+\Gamma-\epsilon\rdgamma$-MMP with respect to morphism $f$ and obtain the following commutative diagram:

\[\begin{tikzcd}
	{(Z,\Gamma)} && {(Z',\Gamma')} & \\
	&&& {(V,P)} \\
	R
	\arrow["p", dashed, from=1-1, to=1-3]
	\arrow["f"', from=1-1, to=3-1]
	\arrow["u", from=1-3, to=2-4]
	\arrow["g", from=2-4, to=3-1]
\end{tikzcd}\]
where $p$ is a $B$-bimeromorphic map defined by the composition of a sequence of flips and divisorial contractions, and $u$ is the Mori fiber space contraction such that $\dim V=2$. Note that we are in one of the following three cases:
\begin{enumerate}
    \item If $\dim R=3$, then $K_Z+\Gamma$ is big, hence $Z$ is Moishezon. Since $Z$ has rational singularities, it follows that $Z$ is projective by \cite[Theorem 1.6]{Nam02}. Since the pluricanonical representation is finite for projective lc pair by \cite[Theorem 3.5]{Fu00}, $\rdgamma$ has no bad component. This is a contradiction.
    \item  If $\dim R=2$, then $g:V\to R$ is a proper bimeromorphic morphism, which implies that $V$ is Moishezon, as $R$ is projective. Since $u$ is a projective morphism and $V$ is Moishezon, it follows that $Z'$ is Moishezon. This implies that $Z$ is Moishezon. Since $Z$ has rational singularities, it is again projective by \cite[Theorem 1.6]{Nam02}. This is again a contradiction.
    \item  Assume that $\dim R=1$. Recall that in this case there are two possibilities:
    \subitem (a) $\lfloor \Gamma'\rfloor$ contains exactly one horizontal component, say $S'$, over $V$ satisfying the following commutative diagram:
   
\[\begin{tikzcd}
	{S'} && {S'} \\
	& V
	\arrow["i", dashed, from=1-1, to=1-3]
	\arrow["{u|_{S}}"', from=1-1, to=2-2]
	\arrow["{u|_{S'}}", from=1-3, to=2-2]
\end{tikzcd}\]
where $i$ is a $B$-bimeromorphic involution over $V$. Let $s\in \MA(\lfloor \Gamma \rfloor, 2m(K_Z+\Gamma)|_{\lfloor \Gamma \rfloor})$. By the isomorphism $H^0(\rdgamma,2m(K_Z+\Gamma)|_{\rdgamma})\cong H^0(\lfloor \Gamma'\rfloor, 2m(K_{Z'}+\Gamma')|_{\lfloor \Gamma'\rfloor})$, $s$ is also an element of $H^0(\lfloor \Gamma'\rfloor, 2m(K_{Z'}+\Gamma')|_{\lfloor \Gamma'\rfloor})$. Since $s$ is minimally admissible, from Definition \ref{def:ma-mpa} it follows that $s|_{{\lfloor \Gamma' \rfloor}^h}=s|_{S'}$ is invariant under the involution $i$ over $V$. Then by an identical argument as in the proof of Case II of Theorem \ref{imp} we have (B) in this case.

\subitem (b)  $\lfloor \Gamma'\rfloor$ contains two horizontal components, say $S_1'$ and $S_2'$, over $V$ satisfying the following commutative diagram:

\[\begin{tikzcd}
	{S'_1} && {S_2'} \\
	& V
	\arrow["\cong", from=1-1, to=1-3]
	\arrow["{u|_{S'_1}\cong}"', from=1-1, to=2-2]
	\arrow["{u|_{S'_2} \cong}", from=1-3, to=2-2]
\end{tikzcd}\]
Assume that $D_1\subset Z$ and $D_2\subset Z$ are the strict transform of $S'_1$ and $S_2'$, respectively.
By Remark \ref{rmk:bad} we know that either both of them are bad or both of them are good.
First assume that both of them are good. Let $s\in \MA(\lfloor \Gamma \rfloor, 2m(K_Z+\Gamma)|_{\lfloor \Gamma \rfloor})$. From Definition \ref{def:ma-mpa}(1) it is clear that $s$ is invariant under the $B$-bimeromorphic map $(u|_{S_2'})^{-1}\circ (u|_{S_1'}) :S_1' \to S_2'$. Then by an identical argument as in the proof of Case II of Theorem \ref{imp} we have (B) in this case.

If both of them are bad, then by adjunction we have $K_{D_i}\sim_{\mathbb Q} (K_Z+\Gamma)|_{D_i}\sim_{\mathbb Q} f^*\mathcal O_R(1)$ for $i=1,2$ (since $D_i$ is bad, no boundary divisor occurs after adjunction). Since $D_i$  dominates $R$ and the Kodaira dimension $\kappa(K_{D_i})=0$, this yields a contradiction.
\end{enumerate}

Let $s\in \MA(\rdgamma,2m(K_Z+\Gamma)|_{\rdgamma})$. Then $s|_{\lfloor\Gamma_i\rfloor}$ is minimally admissible on $\lfloor\Gamma_i\rfloor$. Since $\MPA(Z_i,2m(K_{Z_i}+\Gamma_i))\to \MA(\lfloor \Gamma_i\rfloor, 2m(K_{Z_i}+\Gamma_i)|_{\lfloor\Gamma_i\rfloor})$ is surjective, we obtain $s_i\in \MPA(Z_i,2m(K_{Z_i}+\Gamma_i))$ such that $\widetilde{s}=(s_1,....,s_k)\in \MPA(Z,2m(K_Z+\Gamma))$ satisfying $\tilde{s}|_{\rdgamma}=s$. Therefore, it follows that $\MPA(Z,2m(K_Z+\Gamma))\to \MA(\rdgamma,2m(K_Z+\Gamma)|_{\rdgamma})$ is surjective and $\MPA(Z,2m(K_Z+\Gamma))$ generates $\mathcal{O}_{Z}(2m(K_Z+\Gamma))$.
\end{proof}

 The following lemma shows that the above Proposition may fail in case of log-Kodaira dimension zero.
 \begin{lemma}\label{lem: bad0}
     Let $(Z,\Gamma)$ be an irreducible dlt pair such that $K_Z+\Gamma$ is nef and $\kappa(Z,K_Z+\Gamma)=0$. Then for sufficiently large and divisible $m>0$, one of the following holds:
     \begin{enumerate}
         \item $\rdgamma$ is connected and $\MPA(Z,2m(K_Z+\Gamma)) \to \MA(\lfloor \Gamma \rfloor, 2m(K_Z+\Gamma)|_{\lfloor \Gamma \rfloor})$ is surjective.
         \item $\rdgamma$ is disconnected.
          \begin{enumerate}
              \item Every component of $\rdgamma$ is good and $\MPA(Z,2m(K_Z+\Gamma)) \to \MA(\lfloor \Gamma \rfloor, 2m(K_Z+\Gamma)|_{\lfloor \Gamma \rfloor})$ is surjective.
              \item $\rdgamma= S_1 \sqcup S_2$ where both $S_i$ are bad and $\MPA(Z,2m(K_Z+\Gamma)) \to \MA(\lfloor \Gamma \rfloor, 2m(K_Z+\Gamma)|_{\lfloor \Gamma \rfloor})$ may not be surjective.
          \end{enumerate}
     \end{enumerate}
 \end{lemma}
 \begin{proof}
 \begin{enumerate}
     \item If $\rdgamma$ is connected, then we have two subcases.
     \begin{enumerate}
         \item If $(Z,\Gamma)$ is not a bad pair,i.e. each component of $\rdgamma$ is good, then $\MA(\rdgamma, 2m(K_Z+\Gamma)|_{\rdgamma})\cong \A(\rdgamma, 2m(K_Z+\Gamma)|_{\rdgamma})$ and hence we are done by Case I(a) of Theorem \ref{imp}.
         \item If $(Z,\Gamma)$ is a bad pair, we are done by similar argument as in Case I(a) of Theorem \ref{imp} replacing $\A(\rdgamma, 2m(K_Z+\Gamma)|_{\rdgamma})$ by $\MA(\rdgamma, 2m(K_Z+\Gamma)|_{\rdgamma})$ and $\PA(Z,2m(K_Z+\Gamma))$ by $\MPA(Z,2m(K_Z+\Gamma))$.
     \end{enumerate}
     
\item If $\rdgamma$ is disconnected, we have two subcases.
\begin{enumerate}
    \item Each component of $\rdgamma$ is good. Then we have $\MA(\rdgamma, 2m(K_Z+\Gamma)|_{\rdgamma})\cong \A(\rdgamma, 2m(K_Z+\Gamma)|_{\rdgamma})$. Hence the statement follows from Theorem \ref{imp}.
    \item At least one component of $\rdgamma$ is bad.
    Recall that  by a similar argument as in the Case III of Theorem \ref{imp} we can run a $K_{Z}+\Gamma-\epsilon \rdgamma$-MMP and obtain  a Mori fiber space with base of dimension 2. We have the following diagram :

\[\begin{tikzcd}
	{(Z,\Gamma)} && {(Z',\Gamma')}& \\
	&&& {(V,P)}
	\arrow["p", dashed, from=1-1, to=1-3]
	\arrow["u", from=1-3, to=2-4]
\end{tikzcd}\]
where $p$ is the composite of a sequence of flips and divisorial contractions and $u$ is the Mori fiber space contraction. Recall that from Proposition \ref{Pro},
in this case we know that $\lfloor \Gamma'\rfloor$ has exactly two horizontal components with respect to $u$, say $\bar{S}_1$ and $\bar{S}_2$. $\bar{S}_1$ and $\bar{S}_2$ are contained in two distinct connected components. Recall that $(u|_{\bar{S_2}})^{-1}\circ (u|_{\bar{S_1}}):\bar{S_1} \to \bar{S_2}$ is an isomorphism. We claim that $\lfloor \Gamma'\rfloor$ has no vertical component  with respect to $u$. To the contrary assume that $D$ is a vertical component of $\lfloor \Gamma'\rfloor$ with respect to $u$. Since $u$ is a contraction of an extremal ray, by \cite[Proposition 3.1]{CHP16}, there is an effective $\mathbb Q$-Cartier divisor $Q$ on $V$ such that $D=u^*Q$. This implies that $D$ intersects every horizontal divisor positively. This implies that $\Supp D \cup \Supp \bar{S_1}\cup \Supp\bar{S_2}$ is connected. Hence, $\lfloor \Gamma' \rfloor$ is connected, which is a contradiction.
 Thus $\lfloor \Gamma'\rfloor=\bar{S_1}+\bar{S_2}$. Note that in this case $\bar{S_1}$ and $\bar{S_2}$ are two connected components of  $\lfloor \Gamma'\rfloor$. Let $S_1'$ and $S_2'$ be the strict transforms of $\bar{S_1}$ and $\bar{S_2}$ on $X'$, respectively. Observe that we have the following commutative diagram:

\[\begin{tikzcd}
	{S'_1} && {\bar{S_1}} & \\s
	&&& V \\
	{S'_2} && {\bar{S_2}}
	\arrow["p", dashed, from=1-1, to=1-3]
	\arrow["\psi"', dashed, from=1-1, to=3-1]
	\arrow["{u|_{\bar{S_1}}}", from=1-3, to=2-4]
	\arrow["\cong", from=1-3, to=3-3]
	\arrow["p"', dashed, from=3-1, to=3-3]
	\arrow["{u|_{\bar{S_2}}}"', from=3-3, to=2-4]
\end{tikzcd}\]
where $\psi:S'_1 \dashrightarrow S'_2$ is the induced $B$-bimeromorphic map.
\begin{claim}
Both $S'_1$ and $S'_2$ are bad components of $\rdgamma$.
\end{claim}
\begin{proof} 
If possible, assume that $S'_1$ and $S'_2$ are both good components. By Lemma \ref{Conn}, the number of connected components of $\lfloor \Gamma \rfloor$ is $2$. Let $B \subset \Supp(\lfloor \Gamma \rfloor)$ be a bad  component of $\rdgamma$. Then $B$ is contained in one connected component. In particular, $B \cap S'_i\neq \emptyset $ for some $i\in\{1,2\}$ which contradicts the fact that $B$ is bad by Lemma \ref{lem:bad}.
\end{proof}
Thus it follows that $\rdgamma= S_1' \sqcup S_2'$ and both $S_i'$ are bad components of $\rdgamma$. Observe that $\mathcal{O}_{Z}(2m(K_{Z}+\Gamma))\cong \mathcal{O}_{Z}$ (by Theorem \ref{abun}) and $\mathcal{O}_{Z}(2m(K_{Z}+\Gamma))|_{\lfloor \Gamma\rfloor}\cong \mathcal{O}_{S'_1} \oplus \mathcal{O}_{S'_2}$. We have following two cases:
\begin{enumerate}
    \item Assume that $\theta \colon S_1' \dashrightarrow S_2'$ is the $B$-bimeromorphic map induced by the Galois involution. Suppose moreover that the induced pullback map
$
\theta^* \colon H^0(K_{S_2'}) \to H^0(K_{S_1'}), 
\theta^*(t)=\lambda t$ for some $\lambda\in \mathbb{C}$. If $\lambda$
has finite order, then the statement is immediate.
\item Suppose that $S_1'$ and $S_2'$ are not $B$-bimeromorphic induced by the Galois involution. Then $\MA(\rdgamma,2m(K_Z+\Gamma)|_{\rdgamma}\cong H^0(\rdgamma, \mathcal{O}_{\rdgamma})\cong H^0(S_1',\mathcal{O}_{S_1'})\bigoplus H^0(S_2', \mathcal{O}_{S_2'})$. Therefore the restriction map is not surjective.
\end{enumerate}

 \end{enumerate}

\end{enumerate}                                                                                                                                                                                                     
\end{proof}

 The following Corollary shows that slc abundance holds under a restricted hypothesis.
  \begin{corollary}\label{main}
Let $(X,\Delta)$ be a compact K\"{a}hler slc threefold such that $K_X+\Delta$ is nef. Let  $\mu:\sqcup (X_i',\Delta_i'+D_i') \to (X,\Delta)$ be the normalization morphism. Assume that $\kappa(K_{X_i'}+\Delta_i'+D_i')\geq 1$ for all $i$. Then $\mathcal O_X(m(K_X+\Delta))$ is globally generated for a sufficiently large and divisible integer $m$.
\end{corollary}
     \begin{proof}
Recall that $\mu:\sqcup (X_i',\Delta_i'+D'_i) \to (X,\Delta)$ is the normalization morphism and we can take a dlt modification by \cite[Lemma 2.28]{DH25} and replace $(X,\Delta)$ by the resulting dlt pair $(Z,\Gamma):=\sqcup (Z_i,\Gamma_i)$. We have the following two possiblities:
\begin{enumerate}
\item 
If $(Z,\Gamma)$ is a bad pair, then by Proposition \ref{pro:bad-abun}(A), we know that $\MA(\rdgamma,2m(K_Z+\Gamma)|_{\rdgamma})$ generates the line bundle $\mathcal{O}_{\rdgamma}(2m(K_Z+\Gamma)|_{\rdgamma})$ for sufficiently large and divisible $m>0$. Since $\kappa(Z,K_Z+\Gamma)\geq 1$, then by Proposition \ref{pro:bad-abun}(B), we know that $\mbox{MPA}(Z,2m(K_Z+\Gamma))$ generates $\mathcal{O}_{Z}(2m(K_Z+\Gamma))$. Hence, by Lemma \ref{lem:descend}, it follows that $\mathcal{O}_X(m(K_X+\Delta))$ is globally generated.
\item If $(Z,\Gamma)$ is not a bad pair, then by Lemma \ref{lem:admissible}, $\A(\rdgamma,2m(K_Z+\Gamma)|_{\rdgamma})$ generates the line bundle $\mathcal{O}_{\rdgamma}(2m(K_Z+\Gamma)|_{\rdgamma})$. Therefore, by Theorem \ref{imp} and Lemma \ref{end}, it follows that $\mathcal{O}_X(2m(K_X+\Delta))$ is globally generated for a sufficiently large and divisible $m>0$.
\end{enumerate}
\end{proof}

\begin{lemma}\label{lem:descend}
    Let $(X,\Delta)$ be a compact K\"{a}hler slc threefold. $\mu:X' \to X$ be the normalization morphism such that $K_{X'}+\Delta'+D'\sim_{\mathbb{Q}}\mu^*(K_X+\Delta)$. Let $\phi:(Z,\Gamma) \to (X',\Delta'+D')$ be its dlt modification and  $K_Z+\Gamma \sim _{\mathbb{Q}}\phi^*(K_{X'}+\Delta'+D')$. Then every section in $\mbox{MPA}(Z,m(K_Z+\Gamma))$ descends to a section on $(X,\Delta)$.
\end{lemma}
\begin{proof}
    Observe that $H^0(Z,m(K_Z+\Gamma)\cong H^0(X',m(K_{X'}+\Delta'+D'))$. Let $s\in \mbox{MPA}(Z,m(K_Z+\Gamma)) $. We can view $s$ inside $H^0(X',m(K_{X'}+\Delta'+D'))$. Since $s$ is minimally preadmissible, $s|_{D^n}$ is the Galois involution $\tau$-invariant. Then by the same argument as \cite[Lemma 4.9]{Xu19}, we obtain our desired result.
\end{proof}
\section{Abundance for non numerically trivial slc threefold}\label{sec:mainthm}
In this section, we prove abundance holds for slc compact K\"ahler threefolds when $K_X+\Delta$ is not numerically trivial.\\
Recall that $\mu:(X',\Delta'+D'):=\sqcup_{i=1}^{k}(X_i',\Delta_i'+D'_i)\to (X,\Delta)$ is the normalization morphism, where $D'$ is the conductor divisor on $X'$. Assume that $\nu:D^n \to D'$ is the normalization. Then there is a Galois involution $\tau:D^n \to D^n$ over $D'$. Let $\phi:(Z,\Gamma):=\sqcup_{i=1}^{k} (Z_i,\Gamma_i) \to (X',\Delta'+D')$ be a dlt modification of $(X', \Delta'+D')$. \\

We now introduce below the notion of a Dual graph, which will be used in the proof of our main Theorem \ref{thm:mainimp}.
We define a Dual graph \(G_Z\) associated to \((X,\Delta)\) as follows.

\begin{enumerate}
    \item[(i)] The vertices of \(G_Z\) are the irreducible components \(Z_i\) of \(Z\).

    \item[(ii)] If $P$ is a prime divisor contained in $X_i \cap X_j,\; i\neq j$, and $S\subset \Supp \lfloor\Gamma_i \rfloor\subset Z_i$ and $T\subset \Supp \lfloor\Gamma_j \rfloor\subset Z_j$ are the unique prime divisors such that $(\mu\circ\phi)(S)=P=(\mu\circ \phi)(T)$, then we join the vertices corresponding to $Z_i$ and $Z_j$ by an \emph{edge}. We label this edge by  $[S\;|\;T]$.

    \item[(iii)] We do not incorporate any data in this graph which creates a \emph{loop} (a vertex joined to itself by an edge) in the graph. This can happen only in the following scenario: there is a prime component $P$ of the conductor divisor $D$ contained in a unique component $X_i$ of $X$ such that $\phi^{-1}\mu^{-1}(P)$ has two irreducible components in $Z_i$.

\end{enumerate}

Thus the graph records how different irreducible components are glued together through the conductor divisors. Note that if $(X,\Delta)$ is connected, then the associated Dual graph $G_Z$ is also a connected graph.
\begin{theorem}\label{thm:abun-non-numerical}
    Let $(X,\Delta)$ be a connected compact slc threefold such that $K_X+\Delta$ is nef but not numerically trivial. Then $K_X+\Delta$ is semiample.
\end{theorem}
\begin{proof}
 Observe that $K_X+\Delta$ is numerically trivial if and only if $K_{Z_i}+\Gamma_i\equiv 0$ for all $i$. Now from Corollary \ref{main}, we know that if the Kodaira dimension $\kappa(Z_i, K_{Z_i}+\Gamma_i) >0$ for all $i$, then $K_X+\Delta$ is semiample. Thus it is enough to prove that if $\kappa(Z_i, K_{Z_i}+\Gamma_i) >0$ for some $i$ (and hence $K_X+\Delta$ is not numerically trivial), and $\kappa(Z_j, K_{Z_j}+\Gamma_j)=0$ for some $j\neq i$, then $K_X+\Delta$ is semiample.
 Since abundance holds for lc pairs in dimension $3$, $\nu(Z_i, K_{Z_i}+\Gamma_i)=\kappa (Z_i,K_{Z_i}+\Gamma_i)$ for all $i$. 
Recall that since the log abundance theorem is known for lc pairs in dimension $3$ (see Theorem \ref{abun}), it follows that the section ring $R(Z, K_Z+\Gamma)=\oplus_{t\geq } H^0(t(K_Z+\Gamma))$ is a finitely generated $\mathbb{C}$-algebra. Thus there is a sufficiently divisible positive integer $d>0$ such that the section ring of $\mathcal{O}_Z(d(K_Z+\Gamma))$ is generated by $H^0(Z, \mathcal{O}_Z(d(K_Z+\Gamma)))$ over $\mathbb{C}$.

By Lemma \ref{lem:descend}, it is enough to show that $\MPA(Z,m(K_Z+\Gamma))$ generates the line bundle $\mathcal{O}_Z(m(K_Z+\Gamma))$ for sufficiently large multiple $m$ of $d$. 
\begin{claim}\label{clm:generation-on-irreducible-components}
    $\MPA(Z_i,m(K_{Z_i}+\Gamma_i))$ generates the line bundle $\mathcal{O}_{Z_i}(m(K_{Z_i}+\Gamma_i))$ for every $i\in \{1,2,...,k\}$.
\end{claim}

\begin{proof}[Proof of Claim \ref{clm:generation-on-irreducible-components}]
We split the proof into $3$ main cases according to the log-Kodaira dimension of the pair $(Z_i,\Gamma_i)$.
\begin{enumerate}
    \item If $\kappa(Z_i, K_{Z_i}+\Gamma_i)\geq 1$, then the claim follows directly from Proposition \ref{pro:bad-abun} applying to the irreducible pair $(Z_i, \Gamma_i)$.
    \item If $\kappa(Z_i, K_{Z_i}+\Gamma_i)=0$, then by Lemma \ref{lem: bad0}, we have the following subcases.
    \begin{enumerate}
    \item 
   If $\lfloor \Gamma_i \rfloor$ is connected, then by Lemma \ref{lem: bad0} (1), every minimally admissible section on $\lfloor \Gamma_i \rfloor$ lifts to minimally preadmissible section on $Z_i$ and generates the line bundle $\mathcal{O}_{Z_i}(m(K_{Z_i}+\Gamma_i))$.
    \item If $ \lfloor \Gamma_i \rfloor $ is disconnected,  then we have $2$ subcases:
    \begin{enumerate}
     \item Each component of $\lfloor \Gamma_i \rfloor$ is  good (see Definition \ref{def:bad}). In this case the minimally admissible sections on $\lfloor \Gamma_i\rfloor $ lift to minimally preadmissible sections on $Z_i$ by Lemma \ref{lem: bad0}, and hence, generate the line bundle $\mathcal{O}_{Z_i}(m(K_{Z_i}+\Gamma_i))$.
    \item At least one irreducible component of $\lfloor \Gamma_i\rfloor$ is bad. Then, we have $2$ subcases.
    \begin{enumerate}
    \item
    $\lfloor \Gamma_i \rfloor= S_1 \sqcup S_2$, and $S_1$ and $S_2$ are both bad (see Definition \ref{def:bad}) and there does not exist any $B$-bimeromorphic map between $S_1$ and $S_2$ induced by the Galois involution. In this case the claim is immediate, as $H^0(Z_i, 2m(K_{Z_i}+\Gamma_i))=\MPA(Z_i,2m(K_{Z_i}+\Gamma_i))$
    \item  $\lfloor \Gamma_i \rfloor= S_1 \sqcup S_2$, and both $S_1$ and $S_2$ are bad, and there is a  $B$-bimeromorphic map between them induced by the Galois involution. In this case, since $X$ is connected, there exists a component $X_j$ for some $j\neq i$ such that $X_j\cap X_i\neq \emptyset$. Hence, there exists another irreducible component $S$ which occurs in the support of $\lfloor \Gamma_i \rfloor$. Therefore $\lfloor \Gamma_i\rfloor $
has at least three irreducible components which contradicts the assumption. Thus this case cannot occur.\\
\end{enumerate}

\end{enumerate}
\end{enumerate}

\end{enumerate}
\end{proof}
Our goal is to construct global minimally preadmissible sections that generates the line bundle. The construction is not immediate as $\MPA(Z,2m(K_Z+\Gamma)) \to \MA(\lfloor \Gamma \rfloor, 2m(K_Z+\Gamma)|_{\lfloor \Gamma \rfloor})$ may not be surjective.(See Lemma \ref{lem: bad0})\\

But before that we will begin with a definition which will simplify our construction of compatible sections.
\begin{definition}
    Let $(Z,\Gamma)=\sqcup_{i=1}^k (Z_i, \Gamma_i)$ be a dlt pair defined as above. Let $(Z_p, \Gamma_p)$ and $(Z_q,\Gamma_q)$ be two irreducible components of $(Z,\Gamma)$ where $p\neq q\in \{1,2,..,k\}$. Let $S\subset \Supp \lfloor \Gamma_{p} \rfloor$ and $T\subset \Supp\lfloor \Gamma_{q}\rfloor$ be two bad components of the pair $(Z_p,\Gamma_{p})$ and $(Z_q, \Gamma_{q})$ respectively. 
We say $(S;Z_{p}, \Gamma_{p})$ is $\textit{specially connected}$  to $(T; Z_{q}, \Gamma_{q})$  if there exists a simple path (a path which contains no cycle) of the Dual graph $G_Z$ defined above joining $Z_{p}$ and $Z_{q}$ such that the following holds:
\[
\begin{tikzcd}
    {Z_{p}} & {V_1} & {V_2} & \bullet & \bullet & {V_g} & {Z_{q}}
    \arrow["{[S|S_1]}", no head, from=1-1, to=1-2]
    \arrow["{[T_1|S_2]}", no head, from=1-2, to=1-3]
    \arrow[no head, from=1-3, to=1-4]
    \arrow[no head, from=1-4, to=1-5]
    \arrow[dashed, no head, from=1-5, to=1-6]
    \arrow["{[T_g|T]}", no head, from=1-6, to=1-7]
\end{tikzcd}
\]
where $V_1,....,V_g \in \{Z_1,...,Z_k\} \setminus \{Z_{p},Z_{q}\}$.
Moreover, we say $S$ is specially connected to $T$ via a $\textit{bad path}$ if log Kodaira dimension of all $V_\alpha=0$ for all $\alpha \in \{1,2,...,g\}$.
\end{definition}
 $Z= Y_1 \sqcup Y_2 \sqcup..\sqcup Y_b \sqcup Z_{i_0}$, where 
 $Z_{i_0}$ is disjoint union of all irreducible components $Z_i$ of $Z$ for which 
the restriction map $\MPA(Z_i,2m(K_{Z_i}+\Gamma_i))\to \MA(\lfloor \Gamma_i\rfloor, 2m(K_{Z_i}+\Gamma_i)|_{\lfloor\Gamma_i\rfloor})$ is surjective. In particular, the map is not surjective for each $Y_i, i \in \{1,2..,b\}$.
Note that by Lemma \ref{lem: bad0} and Proposition \ref{pro:bad-abun}, $\kappa(Y_i, K_{Y_i}+\Gamma_{Y_i})=0$ and $\lfloor \Gamma_{Y_i}\rfloor = S_{Y_i} \sqcup T_{Y_i}$ where both components are bad.

\begin{claim}\label{clm:generation Z_{i_0}}
$\MPA(Z_{i_0}, 2m(K_{Z_{i_0}}+\Gamma_{i_0}))\to \MA(\lfloor \Gamma_{i_0} \rfloor, 2m(K_{Z_{i_0}}+\Gamma_{i_0})|_{\lfloor \Gamma_{i_0} \rfloor})$ is surjective and 
    $\MPA(Z_{i_0}, 2m(K_{Z_{i_0}}+\Gamma_{i_0}))$ generates the line bundle $\mathcal{O}_{Z_{i_0}}(2m(K_{Z_{i_0}}+\Gamma_{i_0}))$.
\end{claim}
\begin{proof}
    Write $(Z_{i_0},\Gamma_{i_0})=\sqcup_{i=1}^{k-b}(Z_i,\Gamma_i)$. By Claim \ref{clm:generation-on-irreducible-components} $\MPA(Z_i,m(K_{Z_i}+\Gamma_i))$ generates the line bundle $\mathcal{O}_{Z_i}(m(K_{Z_i}+\Gamma_i))$ for every $i\in \{1,2,...,k-b\}$. By Proposition \ref{pro:bad-abun}, we know that $\MA(\lfloor \Gamma_{i_0} \rfloor, 2m(K_{Z_{i_0}}+\Gamma_{i_0})|_{\lfloor \Gamma_{i_0} \rfloor})$ generates $\mathcal{O}_{\lfloor \Gamma_{i_0} \rfloor}(2m(K_{\lfloor \Gamma_{i_0} \rfloor }+\Diff (\Gamma_{i_0} -\lfloor \Gamma_{i_0}\rfloor)))$. Since the restriction map is surjective on each irreducible component, global surjectivity of the restriction map and global generation by preadmissible sections follow.
\end{proof}

\textbf{Construction:}\label{construction}
We start with a non-zero section $s\in\MPA(Z_{i_0},2m(K_{Z_{i_0}} +\Gamma_{i_0}))$. Choose an irreducible component $Y_i $ of $Z$ where $i\in \{1,2,,,,,b\}$. Define $\mathscr{L}=2m(K_{Y_i}+\Gamma_{Y_i})$. Note that $\lfloor \Gamma_{Y_i}\rfloor = S_{Y_i} \sqcup T_{Y_i}$ where both components are bad.
We define a section $w\in H^0(\widetilde{\lfloor \Gamma_{Y_i} \rfloor}, \mathscr{L}|_{\widetilde{\lfloor \Gamma_{Y_i} \rfloor}})$ as follows:
\begin{enumerate}
   
      \item  Exactly one component is specially connected to $(S;Z_{j},\Gamma_{j})$ via bad path for some $S\subset \Supp \lfloor \Gamma_{j}\rfloor$.  Without loss of generality, we assume that it is $T_{Y_i}$. In this case, we define $w|_{T_{Y_i}}=w|_{S_{Y_i}}=c_{T_{Y_i}}=\prod_{l=0}^g \lambda_l(m)(s|_S)\in \mathbb{C}$ where the scalar is obtained by 
      \[
\begin{tikzcd}[column sep=small]
{Z_{j}} 
& {V_1} 
& {V_2} 
& \bullet 
& \bullet 
& {V_g} 
& {Y_i} \\
{H^0(mK_S)} 
& {H^0(mK_{S_1})} 
& {H^0(mK_{T_1})} 
& {H^0(mK_{S_2})} 
&& {H^0(mK_{T_g})} 
& {H^0(mK_{T_{Y_i}})}
\arrow["{[S|S_1]}", no head, from=1-1, to=1-2]
\arrow["{[T_1|S_2]}", no head, from=1-2, to=1-3]
\arrow[no head, from=1-3, to=1-4]
\arrow[no head, from=1-4, to=1-5]
\arrow[dashed, no head, from=1-5, to=1-6]
\arrow["{[T_g|T_{Y_i}]}", no head, from=1-6, to=1-7]
\arrow["{\times \lambda_0(m)}", from=2-1, to=2-2]
\arrow["{\times \lambda_1(m)}", from=2-3, to=2-4]
\arrow["{\times \lambda_g(m)}", from=2-6, to=2-7]
\end{tikzcd}
\]
 \item\label{item:imp} Both $T_{Y_i}$ and $S_{Y_i}$ are specially connected to $(T;Z_{j'},\Gamma_{j'})$ and $(S;Z_{j},\Gamma_{j})$ where both $Z_j,Z_{j'}$ (not necessarily distinct) are components of $Z_{i_0}$ respectively via bad paths. In this case, we define 
      $w|_{T_{Y_i}}=c_{T_{Y_i}}$ and $w|_{S_{Y_i}}=c_{S_{Y_i}}$ as above.
 Observe that each $Y_i$ for $i\in \{1,2,,,,,b\}$ must be joined to some $Z_j$ of $Z_{i_0}$ via bad path.
 
 To lift this section to a minimally preadmissible section on $Y_i$, it must satisfy $c_{T_{Y_i}}=c_{S_{Y_i}}$. In particular, the well-definedness of this construction needs to guaranteed, which we provide below, see Claim \ref{clm:well}.
 
  \end{enumerate}

\begin{claim}\label{clm:well}
    The above construction is well defined.
\end{claim}
\begin{proof}[Proof of Claim \ref{clm:well}]
Let
\[
T_{Y_i} \subset \Supp \lfloor \Gamma_{Y_i} \rfloor
\]
be an irreducible component which is specially connected via bad path to a component
\[
S \subset \Supp \lfloor \Gamma_{j} \rfloor.
\]
Assume that the bad path is the following:

\[
\begin{tikzcd}
    {Z_{j}} & {V_1} & {V_2} & \bullet & \bullet & {V_g} & {Y_i}
    \arrow["{[S|S_1]}", no head, from=1-1, to=1-2]
    \arrow["{[T_1|S_2]}", no head, from=1-2, to=1-3]
    \arrow[no head, from=1-3, to=1-4]
    \arrow[no head, from=1-4, to=1-5]
    \arrow[dashed, no head, from=1-5, to=1-6]
    \arrow["{[T_g|T_{Y_i}]}", no head, from=1-6, to=1-7]
\end{tikzcd}
\]
We claim that this path is unique.
Observe that $T_{Y_i}$ is $B$-bimeromorphic, via the Galois involution, to a unique prime divisor $T_g$. Hence the edge $[T_g|T]$ is uniquely determined, and therefore $Y_i$ is connected to a unique vertex $V_g$.
Since $\kappa(V_g,K_{V_g}+\Gamma_{V_g})=0$, we have
\[
\Supp\lfloor\Gamma_{V_g}\rfloor=S_g\sqcup T_g.
\]
The divisor $S_g$ is $B$-bimeromorphic, via the Galois involution, to a unique prime divisor $T_{g-1}$. Consequently, the edge $[T_{g-1}|S_g]$ is uniquely determined, and thus $V_g$ is connected to a unique vertex $V_{g-1}$.
Repeating this argument inductively, we see that each preceding vertex and edge in the sequence is uniquely determined. Therefore the entire path
is unique.
\end{proof}
Now observe that, in order to extend a section $ s\in \MPA\bigl(Z_{i_0},2m(K_{Z_{i_0}}+\Gamma_{i_0})\bigr)$ to compatible sections on each irreducible component $Y_i$ of $Z$  so that $\widetilde{s}$ satisfying $\widetilde{s}|_{Z_{i_0}}$ defines a global minimally preadmissible section on $Z$, the section $s$ must satisfy $\lambda(m)s|{S}=\mu(m)s|{T}$, by \ref{item:imp}.

Let $(S_1,S_2).....(S_{2p-1},S_{2p})$ be $p$ many pairs of bad components of $\lfloor \Gamma_{i_0} \rfloor $ such that we obtain one equation as above for each pair. This gives us a system of homogeneous two-term linear equations.
We want to find a subspace $U_m=\{ s\in \MPA(Z_{i_0},2m(K_{Z_{i_0}}+\Gamma_{i_0}))\;|\; \sum_{1\leq \alpha \leq p, 1\leq \beta\leq 2p}a_{\alpha \beta}(m) s|_{S_\beta}=0\}$ such that $U_m$ generates the line bundle $\mathcal{O}_{Z_{i_0}}(Z_{i_0}, 2m(K_{Z_{i_0}}+\Gamma_{i_0}))$.\\

We know that $\MPA(Z_{i_0}, 2m(K_{Z_{i_0}}+\Gamma_{i_0}))$ generates the line bundle $\mathcal{O}_{Z_{i_0}}(2m(K_{Z_{i_0}}+\Gamma_{i_0}))$ by Claim \ref{clm:generation Z_{i_0}} and $\dim_{\mathbb C} \MPA(Z_{i_0},2m(K_{Z_{i_0}}+\Gamma_{i_0})) \sim O(m^e)$ for some $e\in\mathbb Z^+$ by Lemma \ref{inv}.
Let $f:Z_{i_0}\to R$ be the morphism (not necessarily with connected fibers) induced by the base-point free linear system $\MPA(Z_{i_0},2m(K_{Z_{i_0}}+\Gamma_{i_0}))$. Since $S_\beta$ is bad, $(K_{Z_{i_0}}+\Gamma_{i_0})|_{S_\beta}=K_{S_\beta}=\mathcal{O}_{S_\beta}$. Therefore every $S_\beta$ maps to a point on $R$. For every component $S_\beta$, define $c_\beta:=f(S_\beta)$. \\

\begin{claim}\label{clm: distinct}
    All $c_\beta$'s are distinct.
\end{claim}
\begin{proof}[Proof of Claim \ref{clm: distinct}] 
Note that for the morphism $Z_{i_0}\to R$
defined by the linear system $\MPA(Z_{i_0}, 2m(K_{Z_{i_0}}+\Gamma_{i_0}))$,
we have
$\MPA(Z_{i_0}, 2m(K_{Z_{i_0}}+ \Gamma_{i_0}))\cong H^0(R, A)$ where A is a very ample divisor on $R$. Fix two components $S_1$ and $S_2$. Since both $S_1$ and $S_2$ are specially connected to some components $T_1$ and $T_2$ where $T_j\subset \sum_{i=1}^b \lfloor\Gamma_{Y_i}\rfloor$ for $j=1,2$. This implies that $S_1$ and $S_2$ are not $B$-bimeromorphic induced by the Galois involution. Recall that $\nu_{\lfloor \Gamma_{i_0}\rfloor}:\widetilde{\lfloor \Gamma_{i_0}\rfloor}\to \lfloor \Gamma_{i_0}\rfloor$ is the normalization. Fix two numbers $a_1\neq a_2 \in \mathbb{C}$.  Consider the section $w\in H^0(\widetilde{\lfloor \Gamma_{i_0}\rfloor}, \nu_{\lfloor\Gamma_{i_0}\rfloor}^*(2m(K_{Z_{i_0}}+\Gamma_{i_0})|_{\lfloor \Gamma_{i_0} \rfloor}))$ such that $w|_{S_l}=a_l$ for $l=1,2$, $a_1\neq a_2$ and $w|_T=0$ for all other components of $\widetilde{\lfloor \Gamma_{i_0}\rfloor}$. Note that $w$ is preserved under any self $B$-bimeromorphic involution of bad component, any $B$-bimeromorphic map between good components and also preserved under any $B$-bimeromorphic map induced by the Galois involution. Thus it follows that $w$ descends to a minimally admissible section $s'$ on $\lfloor\Gamma_{i_0} \rfloor$. By Claim \ref{clm:generation Z_{i_0}}, this section $s'$ lifts to a minimally preadmissible section $s$ on $Z_{i_0}$. This implies that there exists $s\in \MPA(Z_{i_0},2m(K_{Z_{i_0}}+\Gamma_{i_0}))$
which takes distinct values on $S_1$ and $S_2$. Hence $c_1\neq c_2$. This proves our claim. \\
\end{proof}
Now we can identify the subspace $U_m$ as $\{t\in H^0(R,A)\;|\; \sum a_{\alpha \beta}(m)t(c_\beta)=0\}$. Fix any point $x\in R$. Let $D_x$ be the closed subset defined by $D_x:=\{c_1,...,c_{2p},x\}$. Consider the long exact sequence in Cohomology, 
\begin{equation*}
    0\to H^0(R, I_{D_x}\otimes A^k) \to H^0(R,A)\to \mathbb{C}^{2p+1}\to H^1(R, I_{D_x}\otimes A^k)\to.
    ..
\end{equation*}
Since $H^1(R, I_{D_x}\otimes A^k)=0$ for all $k\geq k_x$(depending on $x$), we can choose $t_x\in U_{mk_x}$ such that $t_x(x)\neq 0$. In this way, for every $x$, we can choose a section $t_x$ (depending on $x$) such that $t_x(x)\neq 0$. Note that $U_x=\{y\in R\;| \; t_x(y)\neq 0\}$ is an open set. Since $X$ is compact, we can choose $k_0$ to be large enough such that $H^1(R,I_{D}\otimes A^k)=0$ for all $k\geq k_0$ and we can choose $t_x\in U_{mk_0}$ for some fixed $k_0$. Therefore it follows that $U_{mk_0}$ generates the line bundle $\mathcal{O}_{Z_{i_0}}(Z_{i_0},2mk_0(K_{Z_{i_0}}+\Gamma_{i_0}))$. This shows that for sufficiently large and divisible $m$, $U_m$ generates the line bundle $\mathcal{O}_{Z_{i_0}}(2m(K_{Z_{i_0}}+ \Gamma_{i_0}))$. \\
\end{proof}

\begin{claim} 
    $\MPA(Z, 2m(K_Z+\Gamma)$ generates the line bundle $\mathcal{O}_Z(2m(K_Z+\Gamma))$.
\end{claim}
\begin{proof}
    Let $x\in Z_{i_0}$ be an arbitrary point. Since $U_{2m}$ generates the line bundle, we can find a section $s\in U_{2m}$ such that $s(x)\neq 0$. Since every section in $U_m$ satisfies the system of two term linear equations, by construction, we can define sections on each $Y_i$ by using  appropriate scalars. Let $w\in H^0(Z,2m(K_Z+\Gamma))$ be the constructed section extending $s$. We need to verify $w|_{\rdgamma}$ is minimally admissible. Note that from construction it is clear that the section is preserved under any $B$-bimeromorphic maps between any two good components, self $B$-bimeromorphic map, any $B$-bimeromorphic map between bad components induced by the Galois involution. 
    Thus we only need to verify $w|_{\rdgamma}$ is preadmissible on $\rdgamma$. Recall that $\rdgamma=\sqcup \lfloor \Gamma_{Y_i}\rfloor \sqcup \lfloor \Gamma_{{i_0}}\rfloor$. Since $\lfloor\Gamma_{Y_{i}}\rfloor$ is disjoint union of two bad components, there is no boundary curve on the adjunction pair. So it is enough to verify preadmissibility on $\lfloor\Gamma_{i_0}\lfloor$. This is immediate as $U_m \subset \MPA(Z_{i_0}, m(K_{Z_{i_0}}+\Gamma_{i_0}))$. 
Note that on every $Y_i$, only one  nonzero section is enough to generate the line bundle. Fix any $Y_i$. Take any component of $\lfloor \Gamma_{Y_i}\rfloor$, say, $T_{Y_i}$ (without loss of generality) which is specially connected to some $S\subset \lfloor \Gamma_{i_0}\rfloor$. Choose $s\in U_{2m}$ such that $s|_S\neq 0$. Then the section defined by the scalar on $Y_i$ must be non zero. This section generates the line bundle on $Y_i$. Hence there exists a sub collection of preadmissible sections which generates the line bundle $\mathcal{O}_Z(2m(K_Z+\Gamma))$.\\

By Lemma \ref{lem:descend}, we have $K_X+\Delta$ is semiample.
\end{proof}

\begin{lemma}\label{inv}
   Let $X$ be a compact complex analytic space. Let $L$ be a semiample line bundle on $X$. Let $V=\bigoplus_{m\geq0} V_m \subset H^0(X,L^m)$ be a graded sublinear system. Assume $L$ is globally generated by this sublinear system. Then $\dim V_m \sim \dim H^0(X,L^m)$.
\end{lemma}
\begin{proof}
    Let $f$ be the semiample fibration induced by the sublinear system $V_{m_0}$. Let $f:X \to Y$ be the morphism. Then $L^{m_0}\cong f^*\mathcal{O}_Y(1)$. By projection formula $H^0(X,L^{mm_0})=H^0(Y,f_*\mathcal{O}_X\otimes \mathcal{O}_Y(m))$. By Grauert's direct image theorem and using degree of Hilbert-polynomial, we see that $H^0(Y,f_*\mathcal{O}_X(m)\sim O(m^{\dim Y})$. Thus the result follows.
\end{proof}
Now we extend our above result for compact Normalized-K\"ahler slc pair of dimension $3$.
\begin{proof}[Proof of corollary \ref{cor:mainimp}]
Let $(X,\Delta)$ be a Normalized-K\"ahler slc pair of dimension $3$. Recall that $\mu:(X',D'+\Delta') \to (X,\Delta)$ is its normalization. By definition, $X'$ is K\"ahler. Take a $\mathbb{Q}$-factorial dlt modification $Y$ of $X'$. Let $\phi: Y \to X'$ be the modification satisfying $K_Y+\Gamma=\phi^*(K_{X'}+\Delta'+D')$. Now the result follows directly from the proof of Theorem \ref{thm:abun-non-numerical} and Corollary \ref{main}.
\end{proof}

\section{abundance for numerically trivial slc threefold and Proof of the main theorem}\label{sec:main result}
In this section, we discuss the numerically trivial case. More precisely we prove the following :
\begin{theorem}\label{mainthm:num-trivial}
    Let $(X,\Delta)$ be a compact connected K\"ahler slc threefold such that $K_X+\Delta$ is numerically trivial. Then $K_X+\Delta$ is semiample.
\end{theorem}

Before proving Theorem \ref{mainthm:num-trivial}, we prove a few results which we will use in the main proof. 
\begin{lemma}\label{lem:preservation-kahler}
    Let $S$ be a compact K\"ahler lc surface satisfying $K_S \sim _{\mathbb{Q}}0$ . Let $g: S \dashrightarrow S$ be a $B$-bimeromorphic satisfying $a^2>0$ and $g^*a-a \in NS(S)_{\mathbb{R}}$ for some $(1,1)$-class $a$. Then for every sufficiently divisible $m$ such that $mK_S\sim 0$, then the scalar $\lambda$ by which it acts on $H^0(mK_S)$ must be of finite order.
\end{lemma}
\begin{proof}
    Since the finiteness of pluricanonical representation holds for projective varieties, we may assume that $S$ is non projective K\"ahler. Now let $f: T\to S$ be a minimal resolution. Observe that $K_T+\Delta=f^*K_S$ where $\Delta$ is effective. If $\Delta\neq 0$, then $T$ is uniruled and hence projective. Finiteness of the pluricanonical representation for $(T,\Delta)$ implies the finiteness of the pluricanonical represntation of $S$. This is a contradiction. Hence $\Delta=0$ and $K_T=f^*K_S$. Replacing $S$ by $T$, we may assume that $S$ is smooth and $g$ is automorphism by Lemma \ref{terminal}. 
Set 
\[ V =H^{1,1}(S,\mathbb{R}), N =NS(S)_{\mathbb{R}} \subset V\]
By Hodge-index theorem, we know that the intersection form on $V$ has signature $(1, h^{1,1}(S) -1)$. Moreover, note that the intersection form on $N$
is non positive because if there exists any divisor $D$ with positive self intersection, then either $D$ or $-D$ must be a big divisor. This implies that $S$ is Moishezon and hence, projective which gives a contradiction.\\
Let $\lambda $ be the eigen value of $g^*$ on $H^{2,0}(S)$. Since $g$ is an automorphism, it preserves the integral lattice $H^2(S,\mathbb{Z})$/torsion. Therefore $\lambda$ is an algebraic integer. By Kronecker's theorem, it has a Galois conjugate, say $\eta$ having modulus other than $1$. Since $g^*$ induces isometry on $V$, both $\eta$ and $\eta^{-1}$ are eigen values of $g^*|_V$. We have two $(1,1)$ class $v$ and $w$ such that $g^*v=\eta v$ and $g^*w=\eta^{-1}w$
satisfying $v^2=w^2=0$ and $v\cdot w>0$. [Note that $v\cdot w=0$ implies there exists two dimensional subspace $U\subset V$ such that $u^2=0$ for all $u\in U$ which contradicts signature of the intersection form]. Set $E=\mathbb{R}v\bigoplus \mathbb{R}w$. Then we have, $V=E\bigoplus E^{\perp}$. Now write $a=a_E+ (a_E)^{\perp}$. Since $((a_E)^{\perp})^2\leq 0$. This implies that $(a_E)^2>0$. Then both coefficients of $v$ and $w$ are non zero and hence same for $(g^*-1)a_E$. \\
Now, observe that $N$ is $g^*$ invariant and $(g^*-1)a \in N$ by hypothesis.
The projection operator $V\to E$ satisfies some polynomial in $g^*$. Thus it shows that $(g^*-1)a_E\in N$. Choosing appropriate polynomial in $g^*$, we can show that $v,w\in N$. Then $E\subset N$. Now since $(v+w)^2>0$, this contradicts the non positivity of intersection form on $N$. Hence $\lambda$ is a root of unity. 
\end{proof}
\begin{lemma}
    Let $Z$ be a $\mathbb{Q}$-factorial normal compact K\"ahler threefold. Let $f:Z\to V$ be a projective surjective morphism such that general fiber is isomorphic to $\mathbb{P}^1$. Let $T_1$ and $T_2$ be two disjoint divisors such that $f|_{T_i}\to V$ is $B$-bimeromorphic for both $i$. Let $g: Z' \to Z$ be a log resolution such that strict transform of both $T_i$, say $T_i'$ are smooth. Assume $r_i: T \to T_i'$ be a common resolution which resolves the induced $B$
$B$-bimeromorphic map. Then for any class $\alpha\in H^{1,1}(Z', \mathbb{R})$, $r_1^*(\alpha|_{T_1'})-r_2^*(\alpha|_{T_2'})\in NS(T)_{\mathbb{R}}$.
\end{lemma}
\begin{proof}
    First, we can choose a Zariski open subset $U\subset V$ such that $f':=f|_{Z_U} : Z_U:=f^{-1}(U) \to U $ is a morphism of smooth varieties such that each fiber is isomorphic to $\mathbb{P}^1$.
    Note that $T_{1,U}$ and $T_{2,U}$ are both disjoint sections over $U$. We can choose this open set such that $g$ and all surface modifications to obtain $T$ are isomorphic over $U$. Then clearly $Z_U\cong \mathbb{P}_U (E)$ where $E$ is a rank-2 vector bundle over $U$.
Define $\zeta:=c_1(\mathcal{O}_{\mathbb{P}_U(E)}(1))$.
Then we know that $H^2(Z_U, \mathbb{R})= f^*H^2(U,\mathbb{R})\bigoplus \mathbb{R}\zeta.
$.
In particular, $\alpha|_{Z_U}= f^*u+ t\zeta$.
Let $i_1$ and $i_2:U \to \mathbb{P}_U(E)$ be two disjoint sections that correspond to two line quotient of the vector bundle $E$. In particular, $i_1^*\zeta= c_1(L_1)$ and $i_2^*\zeta= c_1(L_2)$
Hence we have, 
\[ i_1^*\alpha-i_2^*\alpha= t(c_1(L_1) - c_1(L_2))\]
Now since two sections are disjoint, we can say $N_{T_2,U/Z_U}=L_1^*\otimes L_2$.
Observe that $mT_2$ is Cartier for some $m>0$ since $Y$ is $\mathbb{Q}$-factorial. Assume that $h_i: T\to T_i$ be the induced morphism. Define $\theta:= \frac{1}{m} h_2^*(\mathcal{O}(mT_2)|_{T_2})\in NS(T)_{\mathbb{R}}$. Over $T_U:=T\times_V U\cong U$, this class $\theta$ is same as $c_1(N_{T_2,U/Z_U}$. Therefore, if 
$\delta:= r_1^*(\alpha|_{T_1'})-r_2^*(\alpha|_{T_2'})$ and $\gamma:= \delta-t\theta$, then $\gamma|_{T_U}=0$.\\
Set $P:= T\setminus T_U$. Write $P$ as a union of curves,say $P=\cup P_i$. We can write cohomology sequence with support 
\[ H^2_P(T, \mathbb{R})\to H^2(T,\mathbb{R})\to H^2(T_U,\mathbb{R})\]
Clearly, $\gamma$ lies in the image of first map since it is in the kernel of second map. Now by Thom isomorphism theorem, image is generated by fundamental classes of $P_i$. In particular, $\gamma=a_i[P_i] \in NS(T)_{\mathbb{R}}$.
Thus we obtain $\delta\in NS(T)_{\mathbb{R}}$.
\end{proof}
The above lemma helps us to prove the following Corollary.
\begin{corollary}\label{cor:num-trivial}
    Let $Z$ be an irreducible component occurring in case \ref{lem: bad0}(2(b)). Let $Z\dashrightarrow Z'$ be the MMP and $u:Z'\to V$ be the Mori fiber space. Let $S_1$ and $S_2$ on $Z$ be the strict transform of two disjoint components $S_1'$ and  $S_2'$. Then the above lemma holds for $S_1$ and $S_2$ as well. We call the $B$-bimeromorphic map $S_1 \dashrightarrow S_2$ induced above is called a standard link map.
\end{corollary}
\begin{proof}
Observe that we can apply the previous lemma on $Z'$ as $u$ is projective and general fiber is isomorphic to $\mathbb{P}^1$. Moreover, we can choose the open set $U$ in $V$ such that the induced bimeromorphic map $Z\dashrightarrow Z'$ is also isomorphism over $U$. We can choose the resolution which resolves the induced bimeromorphic map $Z\dashrightarrow Z'$ and strict transforms of $S_1$ and $S_2$ are smooth. Any additional surface modifications changes the difference by pullback of divisor classes and exceptional divisor classes, all of which lie in $NS_{\mathbb{R}}$. Hence we obtain our required result.
\end{proof}
Now we are ready to prove our main theorem \ref{mainthm:num-trivial} of this section.
\begin{proof}
\textbf{Step 1:}
Recall that $(Z,\Gamma)=\sqcup_{i=1}^k (Z_i,\Gamma_i) \to (X',\Delta'+D')\to(X,\Delta)$ be $\mathbb{Q}$-factorial dlt modification. Since $K_X+\Delta$ is numerically trivial, $\kappa(Z_i, K_{Z_i}+\Gamma_i)=0$ for all $i$. 
\begin{enumerate}
    \item If $X$ is irreducible, then $(Z,\Gamma)$ is a $\mathbb{Q}$-factorial irreducible dlt pair with $\kappa(Z,K_Z+\Gamma)=0$. In this case, we may assume $\rdgamma$ is disjoint union of two bad components by Lemma \ref{lem: bad0} as for other two cases abundance automatically holds.
    \item If $X$ is not irreducible, then we have the following:
    Assume that $\rdgamma$ has at least one bad component. Without loss of generality assume that $\lfloor\Gamma_1\rfloor$ has a bad component, then by Lemma \ref{lem: bad0}(1) and (3), every component of $\lfloor \Gamma_1 \rfloor$ is bad. Since $X$ is connected, $X_1$ must intersect some $X_j$. Then there must be a component of $\lfloor\Gamma_j\rfloor$ which is bad and $B$-bimeromorphic (induced by Galois involution) to some bad component of $\lfloor\Gamma_1\rfloor$. Since $\kappa(Z_j, K_{Z_j}+\Gamma_j)=0$, hence each component of $\lfloor\Gamma_j\rfloor$ is bad. Thus we may assume that  every component of $\rdgamma$ is bad.
    
\end{enumerate}. Moreover, we can choose $m$ such that $\mathcal{O}_{Z_i}((m(K_{Z_i}+\Gamma_i))\cong \mathcal{O}_{Z_i}$. \\
\textbf{Step 2:}
Let $G$ be a multigraph whose vertices are normalization components and edge records identification between prime conductor divisors. More precisely, if $S\subset X_i'$ and $T\subset X_j'$ are identified by Galois involution, then we draw an edge $e_S^T: i \to j$
We draw a self loop on some vertex whenever two conductor divisors get identified on the same normalization component.\\
For any oriented edge $e_S^T: i \to j $ due to identification of bad components , define $a_{e_S^T}\in \mathbb{C}^*$ be the scalar acting on pluricanonical sections. 
Observe that $H^0(m(K_Z+\Gamma))\cong \bigoplus_{i=1}^k H^0(m(K_{Z_i}+\Gamma_i))= \bigoplus_{i=1}^k \mathbb{C}$.
pick any section $s=(c_1, c_2,...,c_k)$. We know the section descends to the base if it is preserved under the Galois involution. Then $s$ descends to base $X$ if and only if $c_i=a_e\cdot c_j$ for any oriented edge $e$. (Instead of $e_{S}^T$, we just use the notation $e$ for simplicity.)\\
The above system of linear equations has a non zero solution if and only if product of all edge scalars is $1$ on every closed chain on the multigraph $G$.\\
\textbf{Step 3:}
Choose any K\"ahler class $\omega $ on $X$. Its pullback on $X_i'$ or on any resolution  is nef and big class. Its restriction on any prime conductor surface remains nef and big. In particular, its restriction is actually a positive- square $(1,1)$- class. 
Take any closed chain in the multigraph $G$. We start with a surface $S$. The chain alternates the conductor identifications $\tau_j$ ( due to Galois involution) and standard link maps $\psi_j$.
\[ S=S_0 \xrightarrow{\tau_1} S_1 \overset{\psi_1}\dashrightarrow S_2 \xrightarrow{\tau_2} \cdots \overset{\psi_{2g}} \dashrightarrow S_{2g}=S\]
Let $g: S\dashrightarrow S$ be the composition.
Observe that restriction of the pullback of the K\"ahler class is preserved under conductor identifications $\tau_i$. Choose a common resolution which resolves the above maps. From Corollary \ref{cor:num-trivial}, the restrictions across $\psi_j$ agree modulo $NS_{\mathbb{R}}$. Any further modifications agree modulo pullback of divisor classes and exceptional divisor classes 
, all of which lie in $NS_{\mathbb{R}}$. Thus we obtain a common smooth model $\bar{S}$ with a class $\alpha$ on $\bar{S}$ which satisfies $\alpha^2>0$ and $g^*\alpha-\alpha \in NS_{\mathbb{R}}(\bar{S})$.
The product of all scalars around this closed chain is precisely the scalar $\lambda$ that acts by a scaler multiplication on pluricanonical sections. By Lemma \ref{lem:preservation-kahler}, $\lambda$ has finite order, say $N$. We obtain such $N$ for every closed chain. Then we replace $m$ by $mM$ where $M$ is the least common multiple of all such $N$ and obtain the system of linear equations has a nonzero solution. Since $\kappa(Z,K_Z+\Gamma)=0$, only one non zero section is enough to generate the line bundle.
This section descends to a section on $X$ which shows that $K_X+\Delta$ is semiample.
\end{proof}
Now we prove our main theorem.
\begin{proof}[Proof of the theorem \ref{thm:mainimp}]\label{proof:mainthm}
    Let $(X,\Delta)$ be a compact K\"ahler slc threefold. It is enough to prove the result for each connected component. Thus we may assume that $X$ is connected. If $K_X+\Delta$ is not numerically trivial, we are done by  Theorem \ref{thm:abun-non-numerical}. If $K_X+\Delta$ is numerically trivial, then we are done by Theorem \ref{mainthm:num-trivial}.
\end{proof}

\section{Counterexample}\label{sec:slc-abundance}
In this section we prove Theorem \ref{counterexample}.\\

Our goal is to construct an explicit counterexample to the failure of abundance  when $(X,\Delta)$ is an irreducible slc pair such that $\kappa(K_{X'}+\Delta'+D')=0$ where $K_{X'}+\Delta'+D'\sim_{\mathbb{Q}}\mu^*(K_X+\Delta)$ and $\mu:X'\to X$ is the normalization morphism. To achieve this, we follow the idea outlined below.\\
We know that a section $s\in H^0(X',m(K_{X'}+\Delta'+D'))\cong \mathbb{C}$ descends to a section inside $H^0(X,m(K_X+\Delta))$ iff $s$ is invariant under the Galois involution associated with the conductor divisor in $X'$. Since the pluricanonical representation is not always finite for compact K\"ahler surface (see Remark \ref{K3}), in particular, there exists an analytic isomorphism $\phi:S \to S$ whose image under representation is infinite. Suppose that the conductor divisor contains such a surface $S$, and that the induced Galois involution is given by this analytic isomorphism $\phi$. Note that, in this case $0\neq s$ is not invariant under $\phi$. Therefore $s$ cannot descend to a section in $H^0(X,m(K_X+\Delta))$. Building on this idea, we start with a normal K\"ahler threefold $X'$. We consider the conductor divisor and define the Galois involution so that the surface $S$ appears in the conductor divisor and the automorphism $\phi$ induces the Galois involution. Then, by using \cite[Corollary 5.33]{Kol13} we construct our desired $X$ and a morphism  $\mu:X' \to X$ such that $\mu$ is the normalization. The detailed construction is explained below.

\begin{proof}[Proof of Theorem \ref{counterexample}]\label{proof:counterexample}
    
Let $S$ be a non-algebraic $K3$ surface such that the pluricanonical representation is infinite (see Remark \ref{K3}), and set
\[
X' := S \times \mathbb{P}^1.
\]
Let $p_1:X'\to S$ and $p_2:X'\to \mathbb{P}^1$ denote the projections.
We define divisors $D_1:=S\times \{0\}, D_2:=S\times \{\infty\}$ and $D':=D_1+D_2$.

Observe that
\[
K_{X'}\sim p_1^*K_S + p_2^*\mathcal{O}_{\mathbb{P}^1}(-2).
\]
Since $K_S \sim 0$ and $D'\sim p_2^*\mathcal{O}_{\mathbb{P}^1}(2)$, it follows that
\[
K_{X'} +D' \sim 0.
\]
In particular, $K_{X'} +D'$ is nef and satisfies
\[
\kappa(X',K_{X'} +D') = 0.
\]

Now let $\sigma$ be an automorphism of $S$ such that its action on $H^0(K_S)$ is multiplication by a scalar $\lambda\in\mathbb C^*$ which is not a root of unity; we note that such an automorphism exists as explained in Remark \ref{K3}. Define an  isomorphism
\[
\phi : S \times \{0\} \longrightarrow S \times \{\infty\}
\]
by
\[
\phi(x,0) = (\sigma(x), \infty).
\]
This induces an involution
\[
\tau : D' \to D'.
\]
where $\tau|_{S \times \{0\}}=\phi$ and $\tau|_{S\times \{\infty\}}=\phi^{-1}$.
The gluing is described by an equivalence relation $R(\tau) \subset X'\times X'$ defined by
\[
R(\tau) = \Delta_{  X'} \cup \Gamma_{\phi} \cup \Gamma_{\phi^{-1}},
\]
where $\Delta_{ X'}$ denotes the diagonal and $\Gamma_{\phi}$, $\Gamma_{\phi^{-1}}$ are the graphs of $\phi$ and $\phi^{-1}$, respectively. Clearly, $R(\tau)$ is a proper closed analytic subset of $X' \times X'$.

In this situation, the following properties hold:
\begin{enumerate}
    \item $(X', D')$ is a compact K\"ahler dlt pair of dimension $3$; in particular, it is lc.
    \item The involution $\tau$ exchanges the two components $S \times \{0\}$ and $S \times \{\infty\}$, and it preserves lc centers.
    \item Each equivalence class with respect to $R(\tau)$ has cardinality at most two.
\end{enumerate}

Therefore,  by \cite[Theorem 7.1]{DG96}, and Theorem \ref{thm:appendix}, the quotient
\[
X := X'/ R(\tau)
\]
exists as $R(\tau)$ is a finite proper equivalence relation. Moreover, by a similar argument as in \cite[Corollary 5.33]{Kol13}, we obtain a morphism $\mu:X'\to X$ such that $K_{X'}+D'=\mu^*K_X$ and $\mu$ is the normalization of $X$.
Since $K_{X'} + D' \sim 0$, it follows that
\[
\mathcal{O}_{X'}(m(K_{X'}+D')) \cong \mathcal{O}_{X'}
\]
for every integer $m > 0$. In particular,
\[
H^0\!\left(X',\, m(K_{X'}+D')\right) \cong \mathbb{C}.
\]

Recall that $D_1$ and $D_2$ are the two irreducible components of $D'$. The gluing is induced by the automorphism $\sigma : S \to S$, which acts on $H^0(K_S)$ by multiplication by a scalar $\lambda$ which is not a root of unity. Consequently, the induced map
\[
\phi^* : H^0(K_{D_2}) \longrightarrow H^0(K_{D_1})
\]
acts by multiplication by $\lambda$.

A section $s \in H^0\!\left(X',\, d(K_{X'}+D')\right)$ descends to a section of $H^0(X, dK_X)$ if and only if its restriction to the conductor divisor $D'$ is invariant under the involution $\tau$, that is
\begin{equation}\label{eqn:invarience}
    \tau^*(s|_{D'}) = s|_{D'}.
\end{equation}

Since $H^0\!\left(X',\, d(K_{X'}+D')\right)$ is one-dimensional for every positive integer $d>0$, every nonzero section restricts to a nonzero constant on each component of $D'$. The invariance condition in \eqref{eqn:invarience} therefore forces that $\lambda^d = 1$. However, this is a contradiction, as $\lambda$ has infinite order according to our construction. Hence, no nonzero section of $H^0\!\left(X',\, d(K_{X'}+D')\right)$ descends to $X$. In particular,
\[
H^0(X, dK_X) = 0 \quad \text{for all } d> 0.
\]
Since $d>0$ is an arbitrary positive integer, this shows that $K_X$ is not semiample, in other words semi-log canonical abundance fails for $(X, 0)$.\\

By Remark \ref{rmk:jiang}, $X$ is not K\"ahler. However, since its normalization  $X'$ is a compact K\"ahler manifold, it follows that $X$ is Normalized-K\"ahler and hence, it belongs to the Fujiki class.\\
\end{proof}

\begin{remark}\label{rmk:jiang}
 Assume that $X$ is K\"ahler. Choose a K\"ahler class $\omega$ on $X$. $\omega'=\mu^*\omega$ is also a K\"ahler class on $X'$ as $\mu$ is finite.  Since $\omega'$ comes from the base $X$, it's restriction on $S$ is invariant under the gluing automorphism $\sigma$. In particular, $\sigma^*(\omega'|_S)=\omega'|_S$. By \cite[Proposition 1.1]{GM17}, $\sigma$ has finite order. In particular, $\lambda$ is a root of unity. This is a contradiction.
\end{remark}

\begin{remark}\label{rmk:slc surface}
Note that the abundance holds for a compact K\"ahler variety of dimension $2$ with slc singularities. To see this, we start with a compact K\"ahler pair $(X,\Delta)$ of dimension $2$ having slc singularities such that $K_X+\Delta$ is nef. Let $(Z,\Gamma)$ be its dlt modification. Since the pluricanonical representation is finite for curves with lc singularities (see Theorem \ref{C}), the condition $(2)$ of Theorem \ref{imp} is satisfied by Lemma \ref{lem:admissible}. Hence, Theorem \ref{imp} applies to the pair $(Z,\Gamma)$. Therefore, it follows that $K_X+\Delta$ is semiample by Lemma \ref{end}.
\end{remark}

\begin{remark}\label{rmk:slc high}
 From Theorem \ref{counterexample}, we observe that counterexamples to slc abundance for varieties in Fujiki class can be explicitly constructed in every dimension $n\geq 3$. We start with an slc threefold $X$ such that $K_X$ is nef but not semiample as in Theorem \ref{counterexample}. Define $Z:=X \times E$, where $E$ is an elliptic curve. We now verify that $Z$ has slc singularities. Note that for any point $(x, e)\in Z$, $\mathcal{O}_{Z,(x,e)}\cong \mathcal{O}_{X,x}\otimes \mathcal{O}_{E,e}$. Since $\mathcal{O}_{X,x}$ is $S_2$ and $\mathcal{O}_{E,e}$ is regular,  $\mathcal{O}_{Z,(x,e)}$ is also $S_2$ by \cite[Theorem 1.6 (b)]{TY03}.\\
Now we will show that $Z$ is normal crossing in co dimension $1$. Recall that the codimension $1$ points of a variety with slc singularities are either regular or have nodal singularities. In case of nodal singularities the complete local ring has the structure: $\widehat{\mathcal{O}_{X,x}}\cong\mathbb{C}[[u,v]]/(uv)$. Since $E$ is smooth, the complete local ring $\widehat{O}_{E,e}\cong \mathbb{C}[[t]]$ at every point $e\in E$. 
Let $(x,e)\in Z$ be a general point of co dimension $1$ subvariety $V\subset Z$. Then we have the following two possibilities:
\begin{enumerate}
    \item If $X$ is smooth in co dimension $1$, then $\widehat{\mathcal{O}_{X,x}}\cong \mathbb{C}[[u]]$, and hence $\widehat{\mathcal{O}_{Z,(x,e)}}\cong \mathbb{C}[[u,t]]$.
    \item If $X$ has nodal singularity in co dimension $1$, then $\widehat{\mathcal{O}_{X,x}}\cong\mathbb{C}[[u,v]]/(uv)$, and hence $\widehat{\mathcal{O}_{Z,(x,e)}}\cong \mathbb{C}[[u,v,t]]/(uv)$.
\end{enumerate} 
Thus $Z$ has slc singularities. Let $p:Z \to X$ be the projection. Note that $K_Z=p^*K_X$, as $K_E\sim 0$, and hence $K_Z$ is nef. We also have $H^0(Z, mK_Z)\cong H^0(X, mK_X)\otimes_{\mathbb{C}} H^0(E, mK_E)\cong H^0(X, mK_X)=0$ for all $m>0$ sufficiently large and divisible. In particular, $K_Z$ is not semiample.\\  
\end{remark}
\appendix
\section{Existence of Pushout in Complex Analytic Category}
The following theorem is an analytic version of Ferrand's existence of pushout diagram in the complex analytic category, see \cite[Proposition 37.67.3]{SP} and \cite[Theorem 7.1]{DF03}. For the convenience of the readers and to avoid any unpleasant surprises having to do with analytic varieties, we include a  complete proof below.
\begin{theorem}\label{thm:appendix}
    Let $X$,$Y$, $Z$ be a complex analytic space, $i:Y \hookrightarrow X$  a closed embedding, and $f:Y \to Z$ a finite morphism. Then the pushout diagram exists in the complex analytic category, i.e., there exist a compact complex analytic space $P$ and morphisms $\alpha:X \to P$ and $\beta:Z \to P$ such that the following diagram commutes.
   
\[\begin{tikzcd}
	Y & X \\
	Z & P
	\arrow["i", hook, from=1-1, to=1-2]
	\arrow["f"', from=1-1, to=2-1]
	\arrow["\alpha", from=1-2, to=2-2]
	\arrow["\beta"', from=2-1, to=2-2]
\end{tikzcd}\]
and satisfies the universal property below: if there exists a compact complex analytic space $P'$ together with morphisms $\alpha':X \to P'$ and $\beta':Z \to P'$ such that the above diagram commutes, then there exists a unique morphism $\theta:P \to P'$ satisfying the following commutative diagram :
\[\begin{tikzcd}
Y & X & \\
Z & P \\
&& {P'}
\arrow["i", hook, from=1-1, to=1-2]
\arrow["f"', from=1-1, to=2-1]
\arrow["\alpha", from=1-2, to=2-2]
\arrow["{\alpha'}", from=1-2, to=3-3]
\arrow["\beta"', from=2-1, to=2-2]
\arrow["{\beta'}"', from=2-1, to=3-3]
\arrow["{ \exists! \theta}"{description}, dashed, from=2-2, to=3-3]
\end{tikzcd}\]

\end{theorem}
\begin{proof}
    We may assume that $Y \subset X$ is a closed analytic subset.
Since $f$ is finite, $f(Y)$ is a closed analytic subset of $Z$. Let $\sim$
be an equivalence relation generated by $y \sim f(y)$ for every $y\in Y$. We define the quotient space $P:=X \sqcup Z/\sim$. Clearly $P$ is Hausdorff.

We define the structure sheaf $\mathcal{O}_P$ on $P$ as follows:
Let $\pi_X:X \to P$ and $\pi_Z:Z \to P$ be the projection maps.
Then for every open set $U\subset P$ define $\mathcal{O}_P(U):=\{(s,t) \in \mathcal{O}_X(\pi_X^{-1}(U))\times \mathcal{O}_{Z}(\pi_Z^{-1}(U)) |    s(y)=t \circ f(y) \mbox{ for all }y \in Y\}$. Clearly $(P,\mathcal{O}_P)$ is a locally ringed space. We need to show that $(P,\mathcal{O}_P)$ is a complex analytic space.\\
Note that $P=(X\setminus Y) \sqcup (Z\setminus f(Y)) \sqcup \{\mbox{glued locus\}}$.
Observe that $(X\setminus Y)$ and $Z\setminus f(Y)$ are open subsets of $X$ and $Z$ respectively, hence it follows that for any point $p\in (X\setminus Y) \sqcup (Z\setminus f(Y))$, its neighborhood is modeled by the local models in \(X\) or \(Z\). Therefore, we only need to verify existence of a local model for the points inside the glued locus.
Let $x\in Y \subset X$, $z'=f(x) \in Z$ and $p \in P$ the corresponding glued point.
For $x$, we can choose a Stein neighborhood $U_X \cong \{f_1(z)=0,f_2(z)=0,\ldots, f_r(z)=0\}\subset \Omega \subset \mathbb{C}^n$ for some open subset $\Omega$ of $\mathbb C^n$. Similarly, for $z'$ we can choose a Stein neighborhood $U_Z\cong \{g_1(w)=0,g_2(w)=0,\ldots, g_s(w)=0\}\subset \Omega' \subset \mathbb{C}^m$.
Shrinking $U_Z$ if necessary, we may assume that $f^{-1}(U_Z)\subset U_X$. Then $U_P=\pi_X(U_X) \cup \pi_Z(U_Z)$ is an open neighborhood of $p$ in $P$.
Since $f:Y \to Z$ is a finite map, $\mathcal{O}_{Y}(f^{-1}(U_Z))$ is a finite $\mathcal{O}_{Z}(U_Z)$-module.
We know that for any Stein space $W$, $W\cong \Spec_{an}(\mathcal{O}_W(W))$, where $\Spec_{an}(\mathcal{O}_W(W))$ denotes the analytic spectrum of the  $\mathbb{C}$-algebra $\mathcal{O}_W(W)$. Thus, we obtain the following:

\[\begin{tikzcd}
{\Spec_{an}(\mathcal{O}_Y(f^{-1}(U_Z))} && {\Spec_{an}(\mathcal{O}_X(U_X))} \\
\\
{\Spec_{an}(\mathcal{O}_Z(U_Z))} && {\Spec_{an}(\mathcal{O}_X(U_X)\times_{\mathcal{O}_Y(f^{-1}(U_Z))} \mathcal{O}_Z(U_Z))}
\arrow["i", hook, from=1-1, to=1-3]
\arrow["f"', from=1-1, to=3-1]
\arrow["\alpha", from=1-3, to=3-3]
\arrow["\beta"', from=3-1, to=3-3]
\end{tikzcd}\] 
We note that the above commutative diagram also appears in Ferrand's construction where Stein spaces are replaced by affine schemes (see \cite[Tag 37.67]{SP} and \cite{DF03}).
Thus, it follows that 
$U_P\cong \Spec_{an}(\mathcal{O}_P(U_P))\cong \Spec_{an}(\mathcal{O}_X(U_X))\times_{\mathcal{O}_Y(f^{-1}(U_Z))} \mathcal{O}_Z(U_Z))$ is a local model defined by a closed analytic subset inside the open set $\Omega \times \Omega' \subset \mathbb{C}^{m+n}$ for the neighborhood $U_P$ of the glued point $p$. Therefore \((P, \mathcal{O}_P)\) is a complex analytic space. The universal property is immediate from the local construction, since it holds on Stein charts and hence globally by gluing, in the same way as in Ferrand's construction.
\end{proof}

\bibliographystyle{hep}
\bibliography{references} 
\end{document}